\documentclass[a4paper,12pt]{amsart}
\usepackage{amssymb}
\usepackage[english]{babel}
\usepackage{color}

\newtheorem{theorem}{Theorem}[section]
\newtheorem{corollary}[theorem]{Corollary}
\newtheorem{lemma}[theorem]{Lemma}
\newtheorem{proposition}[theorem]{Proposition}
\newtheorem{remark}[theorem]{Remark}
\begin{document}
\title[Continuity of derivations]{Continuity of derivations in algebras of locally measurable operators}
\author{A. F. Ber}
\address{Department of Mathematics,National University of Uzbekistan,
Vuzgorodok, 100174, Tashkent, Uzbekistan}
\email{ber@ucd.uz}

\author{V. I. Chilin }
\address{Department of Mathematics, National University of Uzbekistan,
Vuzgorodok, 100174, Tashkent, Uzbekistan}
\email{chilin@ucd.uz}

\author{F. A. Sukochev}
\address{School of Mathematics and Statistics, University of New South Wales, Sydney, NSW 2052, Australia }
\email{f.sukochev@unsw.edu.au}
\date{\today}
\begin{abstract}
We prove that any derivation of the $*$-algebra  $LS(\mathcal{M})$ of all locally measurable operators affiliated with a properly infinite von Neumann algebra
$\mathcal{M}$ is continuous with respect to the local measure topology $t(\mathcal{M})$. Building an extension of a derivation $\delta:\mathcal{M}\longrightarrow LS(\mathcal{M})$ up to a derivation from $LS(\mathcal{M})$ into $LS(\mathcal{M})$, it is further established that any derivation from $\mathcal{M}$ into $LS(\mathcal{M})$ is $t(\mathcal{M})$-continuous.
\end{abstract}
\maketitle

\section{Introduction}

The theory of derivations of various classes of Banach $*$-algebras (e.g.
$C^*$,  $AW^*$ and $W^*$-algebras) is very well developed (see,
for example, \cite{Brat,Sak,Sak1}). It is well known that every
derivation of a $C^*$-algebra is norm continuous and every
derivation of a $AW^*$-algebra (in particular, of a $W^*$-algebra)
is inner  \cite{Olesen,Sak}. The development of the theory of
noncommutative integration, initiated by I. Segal's paper
\cite{Seg} prompted the introduction of numerous non-trivial
$*$-algebras of unbounded operators, which, in a certain sense,
are close to $AW^*$ and $W^*$-algebras. The main interest here is
represented by the $*$-algebra $LS(\mathcal{M})$ (respectively,
$S(\mathcal{M})$) of all locally measurable (respectively,
measurable) operators, affiliated with a $W^*$-algebra (or with a
$AW^*$-algebra) $\mathcal{M}$ and also by the $*$-algebra
$S(\mathcal{M},\tau)$ of all $\tau$-measurable operators from
$S(\mathcal{M})$, where $\tau$ is a faithful normal semifinite
trace on $\mathcal{M}$ \cite{Ne,San}.  The importance of the algebra $LS(\mathcal{M})$ for the theory of unbounded derivations on von Neumann algebras may be seen from the following classical example. Consider the algebra $\mathcal{M}= L_\infty(0,\infty)$ equipped with the semifinite trace given by Lebesgue integration and consider (a partially defined) derivation $\delta=d/dt$ on $\mathcal{M}$. A simple argument shows that the algebra $LS(\mathcal{M})$, which in this case coincides with the space of all measurable complex functions on $(0,\infty)$ is the only natural receptacle of $\delta$. Similar examples can be produced in much more sophisticated circumstances and clearly indicate that the algebra 
$LS(\mathcal{M})$ is the most suitable object for studying unbounded derivations on a given von Neumann algebra $\mathcal{M}$. However, the study of derivations in the setting of $LS(\mathcal{M})$ has been greatly impeded by the fact that it is not a Banach algebra (it is not even a Frechet algebra or locally convex algebra when endowed with its natural topology). An additional difficulty (especially, in comparison with rather well studied algebras $S(\mathcal{M},\tau)$) is represented by the lack of developed analytical techniques in $LS(\mathcal{M})$.   Only recently, in the series of papers \cite{AAK, AK, BdPS, BCS, Ber} meaningful attempts have been made to study the structure of derivations on such algebras. Of particular interest is the problem of identifying the class of von Neumann algebras, for which any derivation of the $*$-algebra $LS(\mathcal{M})$ is inner.
%Thus, we are naturally faced
%with the problem of studying properties of derivations of such
%$*$-algebras. One of the main problems here is the determining
%the class of $*$-algebras $LS(\mathcal{M})$ which admit non-inner
%derivations.
In the setting of commutative $W^*$-algebras
(respectively, commutative $AW^*$-algebras) this problem  is fully
resolved in \cite{BCS} (respectively, in \cite{Kusraev}). In the setting of von Neumann algebras of type $I$, a thorough treatment of this problem may be found in \cite{AAK} and \cite{BdPS}. The papers
\cite{AAK,BCS} contain examples of non-inner derivations of the
$*$-algebra $LS(\mathcal{M})$, which are not continuous with
respect to the topology $t(\mathcal{M})$ of local convergence in
measure on $LS(\mathcal{M})$.  The latter topology is the only topology considered on algebras  $LS(\mathcal{M})$ to date, it  may be also viewed as a noncommutative generalization of the classical topology of convergence in measure on the sets of finite measure in the case when $\mathcal{M}$ is given by the algebra $L^\infty(\Omega,\Sigma,\mu)$, where $(\Omega,\Sigma,\mu)$ is a $\sigma$-finite measure space (in this case the algebra $LS(\mathcal{M})$ coincides with the algebra of all measurable complex functions on $\Omega$). See details in Section 2 below.
On the other hand, it is shown in
\cite{AAK} that in the special case when $\mathcal{M}$ is a
properly infinite von Neumann algebra of type $I$, every
derivation of $LS(\mathcal{M})$ is continuous with respect to the
local measure topology  $t(\mathcal{M})$.  Moreover, all such derivations are inner.  Using a completely different technique, a
similar result was also obtained in \cite{BdPS} under the
additional assumption that the predual space $\mathcal{M}_*$ to
$\mathcal{M}$ is separable. It is of interest to observe that
an analogue of this result (that is the continuity of an arbitrary
derivation of $(LS(\mathcal{M}),t(\mathcal{M}))$) also holds for any
von Neumann algebra $\mathcal{M}$ of type $III$ \cite{AK}. In \cite{AK} the following problem is formulated (Problem 3): Let $\mathcal{M}$ be a  von Neumann algebra  of type $II$ and let $\tau$ be a faithful, normal, semifinite trace on $\mathcal{M}$.  Is any derivation of a $*$-algebra $S(\mathcal{M},\tau)$ equipped with the (classical) measure  topology generated by the trace  necessarily continuous? In \cite{Ber} this problem is solved affirmatively for a properly infinite algebra  $\mathcal{M}$. In view of the example we mentioned above,   a natural  problem (analogous to Problem 3 from  \cite{AK}) is whether any derivation in a $*$-algebra $LS(\mathcal{M})$ is necessarily continuous with respect to the topology $t(\mathcal{M})$, where $\mathcal{M}$ is a properly infinite von Neumann algebra of type $II$. The main results of this paper provide an affirmative solution to this problem. In fact, we establish a much stronger result that any derivation $\delta: \mathcal{A}\longrightarrow LS(\mathcal{M})$, where $\mathcal{A}$ is any subalgebra in $LS(\mathcal{M})$ containing the algebra $\mathcal{M}$, is necessarily continuous with respect to the topology $t(\mathcal{M})$. The proof proceeds in two stages. Firstly, we establish the
 $t(\mathcal{M})$-continuity of any derivation $\delta: LS(\mathcal{M})\longrightarrow LS(\mathcal{M})$ for a properly infinite von Neumann algebra $\mathcal{M}$ (section 3). Then, in Section 4, a special construction of extension of a derivation $\delta: \mathcal{M}\longrightarrow LS(\mathcal{M})$ up to a derivation defined on the whole algebra $LS(\mathcal{M})$ is given (here $\mathcal{M}$ is actually an arbitrary von Neumann algebra). We also hope that our approach to unbounded derivations on $\mathcal{M}$ as well as techniques developed for dealing with locally measurable operators and the topology of local convergence in measure in this paper  are of interest in their own right and may be used elsewhere.

We use terminology and notations from von Neumann algebra theory
 \cite{Sak, Tak} and theory of locally measurable operators from \cite{MCh, San, Yead}.

\section{Preliminaries}

Let $H$ be a Hilbert space, let $B(H)$ be the $*$-algebra of all
bounded linear operators on $H$, and let $\mathbf{1}$ be the
identity operator on $H$. Given a von Neumann algebra
$\mathcal{M}$ acting on $H$, denote by $\mathcal{Z}(\mathcal{M})$
the center of $\mathcal{M}$ and by
$\mathcal{P}(\mathcal{M})=\{p\in\mathcal{M}:\ p=p^2=p^*\}$ the
lattice of all projections in $\mathcal{M}$. Let
$P_{fin}(\mathcal{M})$ be the set of all finite projections in
$\mathcal{M}$.  Denote by $\tau_{so}$ the strong operator topology on $B(H)$, that is the locally convex topology generated by the family of seminorms $p_\xi(x)=\|x\xi\|_H,\xi\in H$, where $\|\cdot\|_H$ is the Hilbert norm on $H$.

A linear operator $x:\mathfrak{D}\left( x\right) \rightarrow
H $, where the domain $\mathfrak{D}\left( x\right) $ of $x$ is a linear
subspace of $H$, is said to be {\it affiliated} with $\mathcal{M}$ if $yx\subseteq
xy$ for all $y$ from the commutant $\mathcal{M}^{\prime }$ of algebra $\mathcal{M}$.

A densely-defined closed linear operator $x$ (possibly unbounded)
affiliated with $\mathcal{M}$  is said to be \emph{measurable}
with respect to $\mathcal{M}$ if there exists a sequence
$\{p_n\}_{n=1}^\infty\subset \mathcal{P}(\mathcal{M})$ such that
$p_n\uparrow \mathbf{1},\ p_n(H)\subset \mathfrak{D}(x)$ and
$p_n^\bot=\mathbf{1}-p_n\in P_{fin}(\mathcal{M})$ for every
$n\in\mathbb{N}$, where
$\mathbb{N}$ is the set of all natural numbers. Let us denote by
$S(\mathcal{M})$ the set of all measurable operators.

Let $x,y\in S(\mathcal{M})$. It is well known that $x+y,\ xy$ and $x^*$ are
densely-defined and preclosed operators. Moreover, the closures
$\overline{x+y}$ (strong sum), $\overline{xy}$ (strong product)
and $x^*$ are also measurable, and equipped with this operations (see \cite{Seg}) $S(\mathcal{M})$ is a unital
$*$-algebra over the field $\mathbb{C}$ of complex numbers. It is clear that $\mathcal{M}$ is a
$*$-subalgebra of $S(\mathcal{M})$.

A densely-defined linear operator $x$ affiliated with
$\mathcal{M}$ is called \emph{locally measurable} with respect to
$\mathcal{M}$ if there is a sequence $\{z_n\}_{n=1}^\infty$ of
central projections in $\mathcal{M}$ such that $z_n \uparrow
\mathbf{1}$ and $z_nx\in S(\mathcal{M})$ for all $n\in\mathbb{N}$.

The set $LS(\mathcal{M})$ of all locally measurable operators
(with respect to $\mathcal{M}$) is a unital $*$-algebra over the
field $\mathbb{C}$ with respect to the same algebraic operations
as in $S(\mathcal{M})$ \cite{Yead} and $S(\mathcal{M})$ is a
$*$-subalgebra of $LS(\mathcal{M})$. If $\mathcal{M}$ is finite,
or if $\dim(\mathcal{Z}(\mathcal{M}))<\infty$, the algebras
$S(\mathcal{M})$ and $LS(\mathcal{M})$ coincide \cite[ Corollary 2.3.5 and Theorem 2.3.16]{MCh}. If von Neumann algebra $\mathcal{M}$
is of type $III$ and $\dim(\mathcal{Z}(\mathcal{M}))=\infty$, then
$S(\mathcal{M})=\mathcal{M}$ and $LS(\mathcal{M})\neq
\mathcal{M}$ \cite[Theorem 2.2.19, Corollary 2.3.15,]{MCh}.

%
%For every $x\in S(\mathcal{Z}(\mathcal{M}))$ there exists a
%sequence $\{z_n\}_{n=1}^\infty\subset
%\mathcal{P}(\mathcal{Z}(\mathcal{M}))$ such that
%$z_n\uparrow\mathbf{1}$ and $z_nx\in \mathcal{M}$ for all
%$n\in\mathbb{N}$. This means that $x\in LS(\mathcal{M})$. Hence,
%$S(\mathcal{Z}(\mathcal{M}))$ is a $*$-subalgebra of
%$LS(\mathcal{M})$, and $S(\mathcal{Z}(\mathcal{M}))$ coincides
%with the center of the $*$-algebra $LS(\mathcal{M})$ {\color{blue} [������] (������ ����� ����������� �� ��������)}

For every subset $E\subset LS(\mathcal{M})$, the sets of all
self-adjoint (resp., positive) operators in $E$ will be denoted by
$E_h$ (resp. $E_+$). The partial order in $LS(\mathcal{M})$ is
defined by its cone $LS_+(\mathcal{M})$ and is denoted by $\leq$.

We shall need the following important property of the $*$-algebra
$LS(\mathcal{M})$. Let $\{z_i\}_{i\in I}$ be a family of pairwise
orthogonal non-zero central projections from $\mathcal{M}$ with
$\sup_{i\in I}z_i=\mathbf{1}$, where $I$ is an arbitrary set of
indices (in this case, the family $\{z_i\}_{i\in I}$ is called a
central decomposition of the unity  $\mathbf{1}$). Consider the
$*$-algebra $\prod_{i\in I}LS(z_i\mathcal{M})$ with the
coordinate-wise operations and involution and set
$$\phi:\ LS(\mathcal{M})\rightarrow
\prod_{i\in I} LS(z_i\mathcal{M}),\ \phi(x):=\{z_ix\}_{i\in I}.
$$

\begin{proposition} \cite{MCh},\cite{Saito}.
\label{Saito} The mapping $\phi$ is a $*$-isomorphism from
$LS(\mathcal{M})$ onto $\prod_{i\in I}LS(z_i\mathcal{M})$.
\end{proposition}

Observe that the analogue of Proposition \ref{Saito} for the
$*$-algebra $S(\mathcal{M})$ does not hold in general
 \cite[\S2.3]{MCh}.

Proposition \ref{Saito}
implies that given any central decomposition $\{z_i\}_{i\in I}$
of the unity and any family of elements $\{x_i\}_{i\in I}$ in
$LS(\mathcal{M})$, there exists a unique element $x\in
LS(\mathcal{M})$ such that $z_ix=z_ix_i$ for all $i\in I$. This
element is denoted by $x=\sum_{i\in I}z_ix_i$.

It is shown in \cite{MCh2} that if $\mathcal{M}$ is of type $I$ or $III$, then for any $x\in LS(\mathcal{M})$ there exists a countable central decomposition of unity
$\{z_n\}_{n=1}^\infty$, such that $x=\sum_{n=1}^\infty z_n x$ and $z_nx\in
\mathcal{M}$ for all $n\in \mathbb{N}$.

Let $x$ be a closed operator with dense domain $\mathfrak{D}(x)$
in $H$, let $x=u|x|$ be the polar decomposition of the operator $x$,
where $|x|=(x^*x)^{\frac{1}{2}}$ and $u$ is a  partial isometry
in $B(H)$ such that $u^*u$ is the right support $r(x)$ of $x$. It is
known that $x\in LS(\mathcal{M})$ (respectively, $x\in S(\mathcal{M})$) if and only if $|x|\in
LS(\mathcal{M})$ (respectively, $|x|\in
S(\mathcal{M})$) and $u\in \mathcal{M}$~\cite[\S\S\,2.2,2.3]{MCh}. If
$x$ is a self-adjoint operator affiliated with $\mathcal{M}$, then
the spectral family of projections $\{E_\lambda(x)\}_{\lambda\in
  \mathbf{R}}$ for $x$ belongs to
$\mathcal{M}$~\cite[\S\,2.1]{MCh}. A locally measurable operator $x$ is measurable if and only if $E_\lambda^\bot(|x|)\in \mathcal{P}_{fin}(\mathcal{M})$ for some $\lambda>0$ \cite[\S\,2.2]{MCh}.

Let us now recall the definition of the local measure topology.
First let $\mathcal{M}$ be a commutative von Neumann algebra. Then
$\mathcal{M}$ is $*$-isomorphic to the $*$-algebra
$L^\infty(\Omega,\Sigma,\mu)$ of all essentially bounded
measurable complex-valued functions defined on a measure space
$(\Omega,\Sigma,\mu)$ with the measure $\mu$ satisfying the direct
sum property (we identify functions that are equal almost
everywhere) (see e.g. \cite[Ch. III, \S 1]{Tak}).  The direct sum property of a measure $\mu$ means
that the Boolean algebra of all projections of the $*$-algebra
$L^\infty(\Omega,\Sigma,\mu)$ is order complete, and for any
non-zero $ p\in \mathcal{P}(\mathcal{M})$ there exists a non-zero
projection $q\leq p$ such that $\mu(q)<\infty$.  The direct sum property of a measure $\mu$ is equivalent to the fact that the functional $\tau(f):=\int_\Omega f\,d\mu$ is a semi-finite normal faithful trace on the algebra $L^\infty(\Omega,\sigma,\mu)$.

Consider the $*$-algebra
$LS(\mathcal{M})=S(\mathcal{M})=L^0(\Omega,\Sigma,\mu)$ of all
measurable almost everywhere finite complex-valued functions
defined on $(\Omega,\Sigma,\mu)$ (functions that are equal almost
everywhere are identified).  On $L^0(\Omega,\Sigma,\mu)$, define
the local measure topology $t(L^\infty(\Omega))$, that is, the
Hausdorff vector topology, whose base of neighborhoods of zero is
given by
$$
  W(B,\varepsilon,\delta):= \{f\in\ L^0(\Omega,\, \Sigma,\, \mu)
  \colon
  \ \hbox{there exists a set} \ E\in \Sigma\
  \mbox{such that}
  $$
  $$
   E\subseteq B, \ \mu(B\setminus
  E)\leq\delta, \ f\chi_E \in L^\infty(\Omega,\Sigma,\mu), \
  \|f\chi_E\|_{{L^\infty}(\Omega,\Sigma,\mu)}\leq\varepsilon\},
$$
where $\varepsilon, \ \delta >0$, $B\in\Sigma$, $\mu(B)<\infty$, and
$$
\chi(\omega)=\left\{\begin{array}{rcl}
1&,& \ \ \omega\in E, \\ 0&,& \ \ \ \omega  \ \not\in \ E.
\end{array}\right.
  $$

Convergence of a net $\{f_\alpha\}$ to $f$ in the topology
$t(L^\infty(\Omega))$, denoted by $f_\alpha
\stackrel{t(L^\infty(\Omega))}{\longrightarrow}f$, means that
$f_\alpha \chi_B \longrightarrow f\chi_B$ in measure $\mu$ for any
$B\in \Sigma$ with $\mu(B)<\infty$. Note, that the topology
$t(L^\infty(\Omega))$ does not change if the measure $\mu$ is
replaced with an equivalent measure \cite{Yead}.

Now let $\mathcal{M}$ be an arbitrary von Neumann algebra and let
$\varphi$ be a $*$-isomorphism from $\mathcal{Z}(\mathcal{M})$
onto the $*$-algebra $L^\infty(\Omega,\Sigma,\mu)$, where $\mu$ is
a measure satisfying the direct sum property.  Denote by
$L^+(\Omega,\, \Sigma,\, m)$ the set of all measurable real-valued
functions defined on $(\Omega,\Sigma,\mu)$ and taking values in
the extended half-line $[0,\, \infty]$ (functions that are equal
almost everywhere are identified). It was shown in~\cite{Seg} that
there exists a mapping
$$
\mathcal{D}\colon
\mathcal{P}(\mathcal{M})\to L^+(\Omega,\Sigma,\mu)
$$
that possesses the following properties:
\begin{itemize}
\item[(D1)]  $\mathcal{D}(p)\in L_+^0(\Omega,\Sigma,\mu)\Longleftrightarrow p\in \mathcal{P}_{fin}(\mathcal{M})$;
\item[(D2)] $\mathcal{D}(p\vee q)=\mathcal{D}(p)+\mathcal{D}(q)$ if
  $pq=0$;
\item[(D3)] $\mathcal{D}(u^*u)=\mathcal{D}(uu^*)$ for any partial
  isometry $u\in \mathcal{M}$;
\item[(D4)] $\mathcal{D}(zp)=\varphi(z)\mathcal{D}(p)$ for any $z\in
  \mathcal{P}(\mathcal{Z}(\mathcal{M}))$ and $p\in
  \mathcal{P}(\mathcal{M})$;
\item[(D5)] if $p_\alpha, p\in
  \mathcal{P}(\mathcal{M})$, $\alpha\in A$ and $p_\alpha\uparrow p$, then
  $\mathcal{D}(p)=\sup\limits_{\alpha\in A}\mathcal{D}(p_\alpha)$.
\end{itemize}

A mapping $\mathcal{D}\colon \mathcal{P}(\mathcal{M})\to
L^+(\Omega,\Sigma,\mu)$ that satisfies properties (D1)---(D5) is
called a \textit{dimension function} on $\mathcal{P}(\mathcal{M})$.

A dimension function $\mathcal{D}$ also has the following properties \cite{Seg}:
\begin{itemize}
\item[(D6)] if $p_n\in\mathcal{P}(\mathcal{M})$, $n\in\mathbb{N}$, then $\mathcal{D}(\sup_{n\geq 1} p_n)\leq\sum_{n=1}^\infty\mathcal{D}(p_n)$, in addition,  when $p_np_m=0$, $n\neq m$, the equality holds;
\item[(D7)] if $p_n\in\mathcal{P}_{fin}(\mathcal{M})$, $n\in\mathbb{N}$, $p_n\downarrow 0$, then $\mathcal{D}(p_n)\rightarrow 0$ almost everywhere.
\end{itemize}

For arbitrary scalars $\varepsilon , \delta >0$ and a set $B\in
\Sigma$, $\mu(B)<\infty$, we set
$$
  V(B,\varepsilon, \delta ) := \{x\in LS(\mathcal{M})\colon \
  \mbox{there exist} \ p\in \mathcal{P}(\mathcal{M}),\
  z\in \mathcal{P}(\mathcal{Z}(\mathcal{M})),
  $$
  $$
   \mbox{such that} \ xp\in \mathcal{M},
  \|xp\|_{\mathcal{M}}\leq\varepsilon,
  \ \varphi(z^\bot) \in W(B,\varepsilon,\delta), \
    \mathcal{D}(zp^\bot)\leq\varepsilon \varphi(z)\},
$$
where $\|\cdot\|_{\mathcal{M}}$ is the $C^*$-norm on $\mathcal{M}$.

It was shown in~\cite{Yead} that the system of sets
\begin{equation*}
%\label{eq_xV}
 \{x+V(B,\,\varepsilon,\,\delta)\colon \ x \in LS(\mathcal{M}),\
 \varepsilon, \ \delta >0,\ B\in\Sigma,\ \mu(B)<\infty\}
\end{equation*}
defines a Hausdorff vector topology $t(\mathcal{M})$ on
$LS(\mathcal{M})$ such that the sets $\{x+V(B,\,\varepsilon,\,\delta)\}$, $\varepsilon, \ \delta >0$, $B\in \Sigma$, $\mu(B)<\infty$ form a neighborhood
base of an operator $x\in LS(\mathcal{M})$. It is known that
$(LS(\mathcal{M}), t(\mathcal{M}))$ is a complete topological
$*$-algebra, and the topology $t(\mathcal{M})$ does not depend on
a choice of dimension function $\mathcal{D}$~ and on the choice of $*$-isomorphism $\varphi$ (see e.g. \cite[\S3.5]{MCh}, \cite{Yead}).

The topology $t(\mathcal{M})$ on $LS(\mathcal{M})$ is called the \textit{local
  measure topology } (or the \textit{topology of convergence locally in measure}).  Note, that in case when $\mathcal{M}=B(H)$ the equality $LS(\mathcal{M})=\mathcal{M}$ holds \cite[\S 2.3]{MCh} and the topology $t(\mathcal{M})$ coincides with the uniform topology, generated by the $C^*$-norm $\|\cdot\|_{B(H)}$.

We will need the following  criterion for convergence of nets from $LS(\mathcal{M})$
with respect to this topology.

\begin{proposition}[{\cite[\S\,3.5]{MCh}}]
\label{plm-spk1}
(i). A net $\{p_\alpha\}_{\alpha\in A}\subset
\mathcal{P}(\mathcal{M})$ converges to zero with respect to the topology $t(\mathcal{M})$ if and only if there is a net $\{z_\alpha\}_{\alpha\in A}\subset \mathcal{P}(\mathcal{Z}(\mathcal{M}))$ such that $z_\alpha p_\alpha\in \mathcal{P}_{fin}(\mathcal{M})$ for all $\alpha\in A$, $\varphi(z^\bot_\alpha)
\stackrel{t(L^\infty(\Omega))}{\longrightarrow} 0$, and $\mathcal{D}(z_\alpha
p_\alpha)\stackrel{t(L^\infty(\Omega))}{\longrightarrow} 0$, where $t(L^\infty(\Omega))$ is the local measure topology  on $L^0(\Omega, \Sigma, \mu)$, and $\varphi$ is a $*$-isomorphism of $\mathcal{Z}(\mathcal{M})$ onto $L^\infty(\Omega,\Sigma,\mu)$.

(ii). A net $\{x_\alpha\}_{\alpha\in A} \subset LS(\mathcal{M})$ converges to zero with respect to the topology $t(\mathcal{M})$ if and only if $E^\bot_\lambda(|x_\alpha|) \stackrel{t(\mathcal{M})}{\longrightarrow} 0$ for every $\lambda>0$, where $\{E_\lambda(|x_\alpha|)\}$ is the spectral projection family for the operator $|x_\alpha|$.
\end{proposition}

\begin{remark}\label{rem_plm-spk1}
It follows from Proposition \ref{plm-spk1} that, if $q_\alpha,p_\alpha\in \mathcal{P(M)},q_\alpha\leq p_\alpha$ and $p_\alpha \stackrel{t(\mathcal{M})}{\longrightarrow} 0$, then $q_\alpha \stackrel{t(\mathcal{M})}{\longrightarrow} 0$.
\end{remark}

Since the involution is continuous in the topology
$t(\mathcal{M})$, the set $LS_h(\mathcal{M})$ is closed in
$(LS(\mathcal{M}),t(\mathcal{M}))$. The cone $LS_+(\mathcal{M})$
of positive elements is also closed in
$(LS(\mathcal{M}),t(\mathcal{M}))$~\cite{Yead}.

Proposition \ref{plm-spk1} will be used in the proof of the following convergence criterion.

\begin{proposition}
\label{p3} If $x_\alpha\in LS(\mathcal{M}),\ 0\neq z \in
\mathcal{P}(\mathcal{Z}(\mathcal{M}))$, then
$$zx_\alpha\stackrel{t(\mathcal{M})}{\longrightarrow} 0\Longleftrightarrow zx_\alpha\stackrel{t(z\mathcal{M})}{\longrightarrow} 0.$$
\end{proposition}
\begin{proof}
Fix a $*$-isomorphism $\varphi: \mathcal{Z}(\mathcal{M})\to L^\infty(\Omega,\Sigma,\mu)$ and $0\neq z\in
\mathcal{P}(\mathcal{Z}(\mathcal{M}))$. Let $E\in
\Sigma$ be such that $\varphi(z)=\chi_E$. Define the mapping
$$\psi:\mathcal{Z}(z\mathcal{M})=z\mathcal{Z}(\mathcal{M})\to
L^\infty(E,\Sigma_E, \mu|_E)$$ by setting
$$
\psi(za)=\varphi(za)|_E, \ \hbox{for} \ a\in \mathcal{Z}(\mathcal{M}).
$$
Here, $\Sigma_E:=\{A\cap E: A\in \Sigma\}$ and $\mu|_E$ is the
restriction of $\mu$ to $\Sigma_E$. It is clear that $\psi$ is a
$*$-isomorphism. Now define
$\mathcal{D}_z:\mathcal{P}(z\mathcal{M})\rightarrow
L_+(E,\Sigma_E, \mu|_E)$ by setting $\mathcal{D}_z(q)=\mathcal{D}(q)|_E$
for $q\in \mathcal{P}(z\mathcal{M})$. It is straightforward that $\mathcal{D}_z$ is a
dimension function on $\mathcal{P}(z\mathcal{M})$.

Let $\{q_\alpha\}_{\alpha\in A}\subset \mathcal{P}(z\mathcal{M})$. We claim
$$q_\alpha \stackrel{t(\mathcal{M})}{\longrightarrow} 0\Longleftrightarrow q_\alpha
\stackrel{t(z\mathcal{M})}{\longrightarrow}0.$$
To see the claim, assume that the first convergence holds and observe that by
  Proposition~\ref{plm-spk1}$(i)$, there exists a net
  $\{z_\alpha\}_{\alpha\in A}\subset
  \mathcal{P}(\mathcal{Z}(\mathcal{M}))$ such that $z_\alpha
  q_\alpha \in \mathcal{P}_{fin}(\mathcal{M})$ for any $\alpha\in
  A$,
  $\varphi(z_\alpha^\bot)\stackrel{t(L^\infty(\Omega))}{\longrightarrow}0$,
  and $\mathcal{D}(z_\alpha
  q_\alpha)\stackrel{t(L^\infty(\Omega))}{\longrightarrow}0$.
  The projection $r_\alpha=zz_\alpha$ belongs to the center
  $\mathcal{Z}(z\mathcal{M})$ of the von Neumann algebra
  $z\mathcal{M}$, and $r_\alpha q_\alpha=z_\alpha q_\alpha$ is a
  finite projection in $z\mathcal{M}$ for each $\alpha\in A$.  Also
  $$
  \psi(z-r_\alpha)=\psi(z(\mathbf{1}-z_\alpha))=\varphi(z z_\alpha^\bot)|_E=\varphi(z)\varphi( z_\alpha^\bot)|_E\stackrel{t(L^\infty(E))}{\longrightarrow}0,
  $$
 where $t(L^\infty(E))$ is the local measure topology on $L^0(E,\Sigma_E,\mu|_E)$, and
 $$
 \mathcal{D}_z(r_\alpha q_\alpha)=\mathcal{D}_z(z_\alpha q_\alpha)=\mathcal{D}(z_\alpha q_\alpha)|_E\stackrel{t(L^\infty(E))}{\longrightarrow}0.
 $$
  Hence, by Proposition~\ref{plm-spk1}(i) we get that $q_\alpha
\stackrel{t(z\mathcal{M})}{\longrightarrow}0$.

We will show now that the convergence $q_\alpha
\stackrel{t(z\mathcal{M})}{\longrightarrow}0$ for $\{q_\alpha\}_{\alpha\in A}\subset \mathcal{P}(z\mathcal{M})$
implies the convergence $q_\alpha\stackrel{t(\mathcal{M})}{\longrightarrow}0$.

Let $\{r_\alpha\}_{\alpha\in A}$ be a net in $\mathcal{P}(\mathcal{Z}(z\mathcal{M}))$ such that $r_\alpha q_\alpha\in \mathcal{P}_{fin}(z\mathcal{M})$
for every $\alpha\in A$,
$$
\psi(z-r_\alpha)\stackrel{t(L^\infty(E))}{\longrightarrow}0
$$
and
$$
\mathcal{D}_z(r_\alpha
q_\alpha)\stackrel{t(L^\infty(E))}{\longrightarrow}0.
$$
Put $z_\alpha=z^\bot+r_\alpha$. Then $z_\alpha\in
\mathcal{P}(\mathcal{Z}(\mathcal{M}))$ and $z_\alpha q_\alpha=
r_\alpha q_\alpha\in\mathcal{P}_{fin}(\mathcal{M}) $. Since
$z^\bot_\alpha=z(\mathbf{1}-r_\alpha)$, we have
$\varphi(z_\alpha^\bot)=\chi_E\varphi(z_\alpha^\bot)$ and
  $$
  \varphi(z^\bot_\alpha)|_E=\chi_E\varphi(z(\mathbf{1}-r_\alpha))|_E=\chi_E\psi(z-r_\alpha)\stackrel{t(L^\infty(E))}{\longrightarrow}0.
  $$
 Also
 $$
 \mathcal{D}(z_\alpha q_\alpha)=\mathcal{D}(zr_\alpha q_\alpha)=\chi_E\mathcal{D}(r_\alpha q_\alpha),
 $$
  and so $\mathcal{D}(z_\alpha q_\alpha)\stackrel{t(L^\infty(\Omega))}{\longrightarrow}0$, since $\mathcal{D}(r_\alpha q_\alpha)|_E=\mathcal{D}_z(r_\alpha q_\alpha)\stackrel{t(L^\infty(E))}{\longrightarrow}0$.
 Again appealing to Proposition~\ref{plm-spk1}(i), we conclude that $q_\alpha\stackrel{t(\mathcal{M})}{\longrightarrow}0$.

Now let $\{x_\alpha\}\subset LS(z\mathcal{M})$ and $x_\alpha
\stackrel{t(\mathcal{M})}{\longrightarrow}0$. By
Proposition~\ref{plm-spk1}(ii), we have that
$E^\bot_\lambda(|x_\alpha|)
\stackrel{t(\mathcal{M})}{\longrightarrow}0$ for any $\lambda>0$,
where $\{E_\lambda(|x_\alpha|)\}$ is the spectral family for
$|x_\alpha|$. Denote by $\{E^z_\lambda(|x_\alpha|)\}$ the family
of spectral projections for $|x_\alpha|$ in $LS(z\mathcal{M})$,
$\lambda>0$. It is clear that
$E_\lambda(|x_\alpha|)=z^\bot+E^z_\lambda(|x_\alpha|)$ and
$E^\bot_\lambda(|x_\alpha|)=z-E^z_\lambda(|x_\alpha|)$ for all
$\lambda>0$.  It follows from above that
$z-E^z_\lambda(|x_\alpha|)
\stackrel{t(z\mathcal{M})}{\longrightarrow}0$ for all $\lambda>0$.
Hence, by Proposition~\ref{plm-spk1}(ii), it follows that
$x_\alpha \stackrel{t(z\mathcal{M})}{\longrightarrow}0$.

The proof of the implication $x_\alpha
\stackrel{t(z\mathcal{M})}{\longrightarrow}0\Longrightarrow x_\alpha \stackrel{t(\mathcal{M})}{\longrightarrow}0$ is similar and therefore omitted.
\end{proof}

The lattice $\mathcal{P}(\mathcal{M})$ is said to have a countable type, if every family of non-zero pairwise orthogonal projections in
$\mathcal{P}(\mathcal{M})$ is, at most, countable.  A von Neumann algebra is said to be $\sigma$-finite, if the lattice $\mathcal{P(M)}$ has a countable type.
 It is shown in \cite[Lemma 1.1]{Seg}
that a finite von Neumann algebra  $\mathcal{M}$ is $\sigma$-finite, provided that the
lattice $\mathcal{P}(\mathcal{Z}(\mathcal{M}))$ of central projections has a countable type.

If $\mathcal{M}$ is a commutative von Neumann algebra and
$\mathcal{P}(\mathcal{M})$ has a countable type, then  $\mathcal{M}$ is
$*$-isomorphic to a $*$-algebra $L^\infty(\Omega,\Sigma,\mu)$ with
$\mu(\Omega)<\infty$. In this case, the topology
$t(L^\infty(\Omega))$ is metrizable and has a base of neighborhoods
of $0$ consisting of the sets $W(\Omega,1/n,1/n),\ n\in\mathbb{N}$.
In addition, $f_n
\stackrel{t(L^\infty(\Omega))}{\longrightarrow}0\Leftrightarrow f_n\longrightarrow 0$ in measure $\mu$, where $f_n,f\in L^0(\Omega,\Sigma,\mu)=LS(\mathcal{M})$.

Let $\mathcal{M}$ be a commutative von Neumann algebra such that
$\mathcal{P}(\mathcal{M})$ does not have a countable type. Denote by
$\varphi$ a $*$-isomorphism from $\mathcal{M}$ on
$L^\infty(\Omega,\Sigma,\mu)$, where $\mu$ is a measure with the
direct sum property. Due to the latter property, there exists a
family $\{z_i\}_{i\in I}$ of non-zero pairwise orthogonal
projections from $\mathcal{P}(\mathcal{M})$, such that $\sup_{i\in
I} z_i=\mathbf{1}$ and $\mu(\varphi(z_i))<\infty$ for all
 $i\in I$, in particular,
$\mathcal{P}(z_i\mathcal{Z}(\mathcal{M}))$ has a countable type.
Select $A_i\in\Sigma$ so that $\varphi(z_i)=\chi_{A_i}$ and set $$\Sigma_{A_i}=\{A\cap A_i:\ A\in\Sigma\},\ \mu_i(A\cap
A_i)=\mu(A\cap A_i),\ i\in I.$$ Let $t(L^\infty(A_i))$ be the local
measure topology on $L^0(A_i,\Sigma_{A_i},\mu_i)$. Since
$\mu_i(A_i)<\infty$, we see that the topology $t(L^\infty(A_i))$ coincides with the topology of convergence in measure $\mu_i$ in
$L^0(A_i,\Sigma_{A_i},\mu_i)$.

\begin{proposition}
\label{p4} For a net $\{f_\alpha\}_{\alpha\in A}$ and $f$ from
$L^0(\Omega,\Sigma,\mu)$ the following conditions are equivalent:

(i). $f_\alpha\stackrel{t(L^\infty(\Omega))}{\longrightarrow} f$ ;

(ii).
$f_\alpha\chi_{A_i}\stackrel{t(L^\infty(A_i))}{\longrightarrow}
f\chi_{A_i}$ for all $i\in I$.
\end{proposition}
\begin{proof}
The implication $(i)\Rightarrow (ii)$ follows from the definitions of topologies
$t(L^\infty(\Omega))$ and $t(L^\infty(A_i))$.

$(ii)\Rightarrow (i)$. It is sufficient to consider the case when $f=0$.

Consider the set $\Gamma$ of all finite subsets $\gamma$
from  $I$ and order it with respect to inclusion. Consider an increasing net $\chi_{D_\gamma}\uparrow \chi_\Omega$ in $L^0_h(\Omega,\Sigma,\mu)$, where
$D_\gamma=\bigcup_{i\in\gamma} A_i,\ \gamma\in\Gamma$. Take an arbitrary neighborhood of zero $U$ (in the topology $t(L^\infty(\Omega))$ ) and select
 $W(B,\varepsilon,\delta)$ in such a way that  $W(B,\varepsilon,\delta)+W(B,\varepsilon,\delta)\subset U$.
Since $\mu(B\cap D_\gamma)\uparrow \mu(B)<\infty$, we can locate such $\gamma_0\in\Gamma$ that $\mu(B\setminus D_{\gamma_0})\leq
\delta$. Hence, $f_\alpha\chi_{\Omega\setminus
D_{\gamma_0}}\in W(B,\varepsilon,\delta)$ for all $\alpha\in A$.

Since
$f_\alpha\chi_{A_i}\stackrel{t(L^\infty(A_i))}{\longrightarrow} 0$
for all $i\in\gamma_0$ and $\gamma_0$ is a finite set, it follows
$f_\alpha\chi_{D_{\gamma_0}}=\sum_{i\in\gamma_0}f_\alpha\chi_{A_i}\stackrel{t(L^\infty(\Omega))}{\longrightarrow}
0$.

Thus, there exists such $\alpha_0\in A$ that
$f_\alpha\chi_{D_{\gamma_0}}\in W(B,\varepsilon,\delta)$ for all $\alpha\geq \alpha_0$.

In particular,
$$f_\alpha=f_\alpha\chi_{D_{\gamma_0}}+f_\alpha\chi_{\Omega\setminus
D_{\gamma_0}}\in
W(B,\varepsilon,\delta)+W(B,\varepsilon,\delta)\subset U,\ \forall\alpha\geq\alpha_0,$$ which implies the convergence
$f_\alpha\stackrel{t(L^\infty(\Omega))}{\longrightarrow} 0$.
\end{proof}

Let us now establish a variant of Proposition \ref{p4} for an arbitrary von Neumann algebra  $\mathcal{M}$.

Let $\varphi$ be a $*$-isomorphism from $\mathcal{Z}(\mathcal{M})$ onto
$L^\infty(\Omega,\Sigma,\mu)$ and let $\{z_i\}_{i\in I}$ be a central decomposition of the unity. As before, we denote $\Gamma$ the directed set of all finite subsets from  $I$. For every
$\gamma\in\Gamma$ we set $z^{(\gamma)}=\sum_{i\in\gamma} z_i$.

Since $\varphi(z_i)=\chi_{A_i}$ for some $A_i\in\Sigma$,
we see that $\varphi(z^{(\gamma)})=\chi_{D_\gamma}$, where
$D_\gamma=\bigcup_{i\in\gamma}A_i$, and, in addition, $z^{(\gamma)} \uparrow
\mathbf{1}$ which implies
$z^{(\gamma)}\stackrel{t(\mathcal{M})}{\longrightarrow} \mathbf{1}$
(see Proposition \ref{plm-spk1} (i) for $p_\alpha=z_\alpha^\bot$ ). As it the proof of
Proposition \ref{p4} for a given $V(B,\varepsilon,\delta)$
we choose $\gamma_0\in\Gamma$ such that
$x(\mathbf{1}-z^{(\gamma_0)})\in V(B,\varepsilon,\delta)$ for every
$x\in LS(\mathcal{M})$. If $x_\alpha\in LS(\mathcal{M})$ and
$x_\alpha z_i\stackrel{t(z_i\mathcal{M})}{\longrightarrow} 0$ for all
 $i\in I$, then by Proposition \ref{p3}, we have $x_\alpha z_i
\stackrel{t(\mathcal{M})}{\longrightarrow} 0$ for all $i\in I$, and so
 $x_\alpha z^{(\gamma_0)}=\sum_{i\in\gamma_0} x_\alpha
z_i\stackrel{t(\mathcal{M})}{\longrightarrow} 0$.

Hence, there exists such $\alpha_0\in A$ that $x_\alpha
z^{(\gamma_0)}\in V(B,\varepsilon,\delta)$ for all
$\alpha\geq\alpha_0$. This means that $$x_\alpha=x_\alpha
z^{(\gamma_0)}+x_\alpha (\mathbf{1}-z^{(\gamma_0)})\in
V(B,\varepsilon,\delta)+V(B,\varepsilon,\delta)\subset
V(B,2\varepsilon,2\delta).$$

The argument above justifies the following result.

\begin{proposition}
\label{p5} Let $\mathcal{M}$ be an arbitrary von Neumann algebra,
$x_\alpha,x\in LS(\mathcal{M})$, $0\neq z_i\in
\mathcal{P}(\mathcal{Z}(\mathcal{M})),\ z_iz_j=0$ when $i\neq j$,
$\sup_{i\in I} z_i=\mathbf{1}$. The following conditions are equivalent:

(i). $x_\alpha\stackrel{t(\mathcal{M})}{\longrightarrow} x$;

(ii). $z_ix_\alpha\stackrel{t(z_i\mathcal{M})}{\longrightarrow}
z_ix$ for any $i\in I$.
\end{proposition}

\begin{remark}
\label{r1}
From Propositions \ref{Saito} and \ref{p5} it follows that the topology $t(\mathcal{M})$ coincides with the Tikhonov product of topologies $t(z_i\mathcal{M}),\ i\in I$.  In particular, the isomorphism $\phi$ from Proposition \ref{Saito} is a topological $*$-isomorphism from $(LS(\mathcal{M}),t(\mathcal{M}))$ onto $\prod_{i\in I}(LS(z_i\mathcal{M}),t(z_i\mathcal{M}))$.
\end{remark}

Let $\{z_i\}_{i\in I}$ be the same as in the assumption of Proposition
\ref{p5}, let $T: LS(\mathcal{M})\rightarrow LS(\mathcal{M})$ be a linear operator such that
 $T(z_ix)=z_iT(x)$ for all $x\in
LS(\mathcal{M}),\ i\in I$. It is clear that $T_{z_i}(y):=T(y)$, $y\in
LS(z_i\mathcal{M})$ is a linear operator acting in
$LS(z_i\mathcal{M})$. Due to Proposition \ref{p5}, the next result follows immediately.

\begin{corollary}
\label{c1} Let $\mathcal{M}$ and let $\{z_i\}_{i\in I}$ satisfy the same assumptions of Proposition
 \ref{p5}, and let $T: LS(\mathcal{M})\rightarrow
LS(\mathcal{M})$ be a linear operator such that
$T(z_ix)=z_iT(x)$ for all $x\in LS(\mathcal{M}),\ i\in I$.
The following conditions are equivalent:

(i). The mapping $T: (LS(\mathcal{M}),t(\mathcal{M}))\rightarrow
 (LS(\mathcal{M}),t(\mathcal{M}))$ is continuous;

(ii). The mapping $T_{z_i}: (LS(z_i\mathcal{M}),t(z_i\mathcal{M}))\rightarrow
(LS(z_i\mathcal{M}),t(z_i\mathcal{M}))$ is continuous for every
$i\in I$.
\end{corollary}

\section{Continuity of derivations in $*$-algebra $LS(\mathcal{M})$}

Let $\mathcal{M}$ be an arbitrary von Neumann algebra, let $\mathcal{A}$ be a subalgebra in $LS(\mathcal{M})$. A linear mapping
 $\delta: \mathcal{A} \rightarrow LS(\mathcal{M})$
is called a \emph{derivation}  on $\mathcal{A}$ with values in $LS(\mathcal{M})$, if
$\delta(xy)=\delta(x)y+x\delta(y)$ for all $x,y\in
\mathcal{A}$. Each element $a\in \mathcal{A}$ defines a
derivation  $\delta_a(x):=ax-xa$ on $\mathcal{A}$ with values in $\mathcal{A}$.
Derivations $\delta_a,a\in\mathcal{A}$ are said to be
\emph{inner derivations} on $\mathcal{A}$.
Since the operation of multiplication is continuous with respect to the topology $t(\mathcal{M})$, it immediately follows that  any inner derivation of  $\mathcal{A}$ is continuous with respect to the topology $t(\mathcal{M})$.

Now, we list a few properties of derivations on $\mathcal{A}$ which we shall need below.

\begin{lemma}
\label{l1} If $\mathcal{P}(\mathcal{Z}(\mathcal{M}))\subset \mathcal{A}$, $\delta$ is a derivation on $\mathcal{A}$
and  $z\in\mathcal{P}(\mathcal{Z}(\mathcal{M}))$, then $\delta(z)=0$ and
$\delta(zx)=z\delta(x)$ for all $x\in \mathcal{A}$.
\end{lemma}
\begin{proof}
We have that
$\delta(z)=\delta(z^2)=\delta(z)z+z\delta(z)=2z\delta(z)$.
Hence, $z\delta(z)=z(2z\delta(z))=2z\delta(z)$, that is
$z\delta(z)=0$. Therefore, we have $\delta(z)=0$. In particular,
$\delta(zx)=\delta(z)x+z\delta(x)=z\delta(x)$.
\end{proof}

Let $\mathcal{A}$ be an $*$-subalgebra in $LS(\mathcal{M})$, let $\delta$ be a derivation on $\mathcal{A}$ with values in $LS(\mathcal{M})$. Let us define a mapping
$$\delta^*: \mathcal{A}\rightarrow
LS(\mathcal{M}),$$
by setting $\delta^*(x)=(\delta(x^*))^*$, $x\in
\mathcal{A}$. A direct verification shows that $\delta^*$ is also a
derivation on $\mathcal{A}$.

A derivation $\delta$ on $\mathcal{A}$ is said to be
\emph{self-adjoint}, if $\delta=\delta^*$. Every
derivation $\delta$ on $\mathcal{A}$ can be represented in the form
$\delta= Re(\delta)+ i Im(\delta)$, where
$Re(\delta)=(\delta+\delta^*)/2,\ Im(\delta)=(\delta-\delta^*)/2i$
are self-adjoint derivations on $\mathcal{A}$.

Since $(LS(\mathcal{M}),t(\mathcal{M}))$ is a topological
 $*$-algebra, the following result holds.

\begin{lemma}
\label{l2} A derivation $\delta: \mathcal{A}\rightarrow
LS(\mathcal{M})$ is continuous with respect to the topology  $t(\mathcal{M})$ if and only if the self-adjoint derivations $Re(\delta)$
and $Im(\delta)$ are continuous with respect to that topology.
\end{lemma}

As we already stated in the introduction, in the special case, when $\mathcal{M}$ is a properly infinite von Neumann algebra of type $I$ or von Neumann algebra of type $III$, any
derivation of the algebra $LS(\mathcal{M})$ is continuous with respect to the topology  $t(\mathcal{M})$  \cite{AK}. The next theorem extends this result to an arbitrary properly infinite von Neumann algebra.

\begin{theorem}
\label{main} If $\mathcal{M}$ properly infinite von Neumann algebras, then any derivation $\delta:
LS(\mathcal{M})\rightarrow LS(\mathcal{M})$ is continuous with respect to the topology $t(\mathcal{M})$ of local convergence in measure.
\end{theorem}
\begin{proof}
By Lemma \ref{l2}, we may assume that $\delta^*=\delta$.
Since  $\mathcal{Z}(\mathcal{M})$ is a commutative von Neumann algebra, there exists a system $\{z_i\},\ i\in I$ of non-zero pairwise orthogonal projections from $\mathcal{Z}(\mathcal{M})$ such that $\sup_{i\in I}z_i=\mathbf{1}$ and the Boolean algebra $\mathcal{P}(z_i\mathcal{Z}(\mathcal{M}))$ has a countable type for all $i\in I$. By Lemma \ref{l1} we have that $\delta(z_ix)=z_i\delta(x)$ for all $x\in LS(\mathcal{M}), i\in I$. Therefore, by
Corollary \ref{c1}, it is sufficient to prove that each derivation $\delta_{z_i}$ is $t(z_i\mathcal{M})$-continuous, $i\in I$. Thus,  we may assume without loss of generality that the Boolean algebra   $\mathcal{P}(\mathcal{Z}(\mathcal{M}))$  has a countable type.

In this case the topology $t(\mathcal{M})$ is metrizable, and the sets
 $V(\Omega,1/n,1/n)$, $n\in\mathbb{N}$ form a countable base of neighborhoods
of $0$; in particular,
$(LS(\mathcal{M}),t(\mathcal{M}))$ is an $F$-space. Therefore it is sufficient to show that the graph of the linear operator  $\delta$ is a closed set.

Arguing by a contradiction, let us assume that the graph of
$\delta$ is not closed. This means that there exists a sequence
$\{x_n\}\subset LS(\mathcal{M})$, such that
$x_n\stackrel{t(\mathcal{M})}{\longrightarrow} 0$ and
$\delta(x_n)\stackrel{t(\mathcal{M})}{\longrightarrow} x\neq 0$.
Recalling that $(LS(\mathcal{M}),t(\mathcal{M}))$ is a topological
$*$-algebra and that $\delta=\delta^*$, we may assume
that  $x=x^*,\ x_n=x_n^*$ for all $n\in\mathbb{N}$. In this case,
$x=x_+-x_-$, where $x_+,x_-\in LS_+(\mathcal{M})$ are respectively
the positive and negative parts of $x$. Without loss of
generality, we shall also assume that $x_+\neq 0$, otherwise,
instead of the sequence $\{x_n\}$ we consider the sequence
$\{-x_n\}$. Let us select scalars $0<\lambda_1<\lambda_2$ so that
the projection
$$p:=E_{\lambda_2}(x)-E_{\lambda_1}(x)$$ does not vanish. We have that
$0<\lambda_1 p\leq pxp=px\leq\lambda_2 p$ and $\|px\|_\mathcal{M}\leq \lambda_2$. Replacing, if necessary, $x_n$ on $x_n/\lambda_1$, we may assume that
\begin{equation}
\label{pxp}
pxp\geq p.
\end{equation}

By the assumption, $\mathcal{M}$ is a properly infinite von Neumann algebra and therefore, there exist
pairwise orthogonal projections
$\{p_m^{(1)}\}_{m=1}^\infty\subset \mathcal{P}(\mathcal{M})$, such that
$\sup_{m\geq 1} p_m^{(1)}=\mathbf{1}$ and $p_m^{(1)}\sim \mathbf{1}$
for all $m\in\mathbb{N}$, in particular, $p\preceq p_m^{(1)}$. Here, the notation $p\sim q$ denotes the
equivalence of projections $p,q\in\mathcal{P}(\mathcal{M})$, and the notation
 $p\preceq q$ means that there exists a projection
$e\leq q$ such that $p\sim e$. In course of the proof of our main result we shall frequently use the following well-known fact: if $p\sim q$ and $z\in\mathcal{P}(\mathcal{Z}(\mathcal{M}))$ then $pz\sim qz$.

For every $m\in\mathbb{N}$ we select a projection $p_m\leq p_m^{(1)}$,
for which $p_m\sim p$ and denote by $v_m$ a partial isometry from
 $\mathcal{M}$ such that $v_m^*v_m=p,\
v_mv_m^*=p_m$. Clearly, we have $p_mp_k=0$ whenever $m\neq k$
and the projection
\begin{equation}
\label{p_0}
p_0:=\sup_{m\geq 1}p_m
\end{equation}
is infinite as the supremum of pairwise orthogonal and equivalent projections. Taking into account that
$$
p_m=v_mpv_m^*\stackrel{(\ref{pxp})}\leq v_mpxpv_m^*=v_mxv_m^*\in p_m\mathcal{M}p_m,
$$
and
$$ \|v_mxv_m^*\|_{\mathcal{M}}=\|v_mpxpv_m^*\|_{\mathcal{M}}\leq\|pxp\|_{\mathcal{M}}\leq \lambda_2,$$
we see that the series $\sum_{m=1}^\infty v_mxv_m^*$ converges  with respect to the topology $\tau_{so}$ to some operator $y\in\mathcal{M}$ satisfying
\begin{equation}
\label{yp_0}
\|y\|_{\mathcal{M}}=\sup_{m\geq 1}\|v_mxv_m^*\|_{\mathcal{M}}\leq
\|pxp\|_{\mathcal{M}},\ \text{and}\ y\geq p_0.
\end{equation}

In what follows, we shall assume that the central support $c(p_0)$ of the projection $p_0$ is equal to $\mathbf{1}$ (otherwise, we replace the algebra
 $\mathcal{M}$ with the algebra $c(p_0)\mathcal{M}$).

Let $\varphi$ be a $*$-isomorphism from $\mathcal{Z}(\mathcal{M})$ onto $L^\infty(\Omega,\Sigma,\mu)$. By the assumption, the Boolean algebra
$\mathcal{P}(\mathcal{Z}(\mathcal{M}))$ has a countable type, and so we may assume that  $\mu(\Omega)=\int_\Omega \mathbf{1}_{L^\infty(\Omega)} \, d\mu =1$, where $\mathbf{1}_{L^\infty(\Omega)}$ is the identity of the $*$-algebra $L^\infty(\Omega,\Sigma,\mu)$. In this case, the countable base of neighborhoods of $0$ in the topology
 $t(\mathcal{M})$ is formed by the sets
$V(\Omega,1/n,1/n),\ n\in\mathbb{N}$.

Recalling that we have  $x_n\stackrel{t(\mathcal{M})}{\longrightarrow} 0$ and $\delta(x_n)\stackrel{t(\mathcal{M})}{\longrightarrow} x$, we obtain
$$v_mx_nv_m^*\stackrel{t(\mathcal{M})}{\longrightarrow} 0,\
\delta(v_m)x_nv_m^*\stackrel{t(\mathcal{M})}{\longrightarrow} 0,\
v_m\delta(x_n)v_m^*\stackrel{t(\mathcal{M})}{\longrightarrow}
v_mxv_m^*$$ when $n\rightarrow\infty$ for every fixed
$m\in\mathbb{N}$.

Fix $k\in\mathbb{N}$, and using the convergence $v_mx_nv_m^*\stackrel{t(\mathcal{M})}\longrightarrow 0$ for $n\longrightarrow 0$, for each $m\in \mathbb{N}$ select an index $n_1(m,k)$ and projections $q^{(1)}_{m,n}\in\mathcal{P}(\mathcal{M})$, $z^{(1)}_{m,n}\in
\mathcal{P}(\mathcal{Z}(\mathcal{M}))$, such that $$\|v_mx_nv_m^*q^{(1)}_{m,n}\|_{\mathcal{M}}\leq 2^{-m}(k+1)^{-1};$$
$$\int_\Omega\varphi(\mathbf{1}-z^{(1)}_{m,n})d\mu\leq 3^{-1}2^{-m-k-1}$$ and $$\mathcal{D}(z^{(1)}_{m,n}(\mathbf{1}-q^{(1)}_{m,n}))\leq 3^{-1}2^{-m-k-1}\varphi(z^{(1)}_{m,n})$$ for all $n\geq n_1(m,k)$.

Similarly, using the convergence $\delta(v_m)x_nv_m^*\stackrel{t(\mathcal{M})}\longrightarrow 0$ (respectively, $v_m\delta(x_n)v_m^*\stackrel{t(\mathcal{M})}\longrightarrow v_mxv_m^*$) for $n\longrightarrow\infty$, for each $m\in\mathbb{N}$ select indexes $n_2(m,k)$ and $n_3(m,k)$ and projections $q^{(2)}_{m,n},q^{(3)}_{m,n}\in\mathcal{P}(\mathcal{M})$, $z^{(2)}_{m,n},z^{(3)}_{m,n}\in
\mathcal{P}(\mathcal{Z}(\mathcal{M}))$, such that $$\|\delta(v_m)x_nv_m^*q^{(2)}_{m,n}\|_{\mathcal{M}}\leq (3(k+1)2^m)^{-1}$$ (respectively, $\|(v_m\delta(x_n)v_m^*-v_mxv_m^*)q^{(3)}_{m,n}\|_{\mathcal{M}}\leq  (3(k+1)2^m)^{-1}$); $$\int_\Omega\varphi(\mathbf{1}-z^{(i)}_{m,n})d\mu\leq 3^{-1}2^{-m-k-1}$$ and $\mathcal{D}(\mathbf{1}-q^{(i)}_{m,n})\leq3^{-1}2^{-m-k-1}\varphi(z^{(i)}_{m,n})$, $i=2,3$, for all $n\geq n_2(m,k)$ (respectively, $n\geq n_3(m,k)$).

Set $n(m,k)=\max_{i=1,2,3} n_i(m,k)$, $z_m=\inf_{i=1,2,3} z^{(i)}_{m,n(m,k)}$, $q_m=\inf_{i=1,2,3} q^{(i)}_{m,n(m,k)}$. Due to the selection of projections $q_m\in\mathcal{P}(\mathcal{M})$, $z_m\in\mathcal{P}(\mathcal{Z}(\mathcal{M}))$ and indexes $n(m,k)$, we have that for each  $m\in\mathbb{N}$ inequalities hold
\begin{itemize}
\item[(A1)] $\|v_mx_{n(m,k)}v_m^*q_m\|_{\mathcal{M}}\leq
2^{-m}(k+1)^{-1}$;
\item[(A2)] $\|\delta(v_m)x_{n(m,k)}v_m^*q_m\|_{\mathcal{M}}\leq
(3(k+1)2^{m})^{-1}$;
\item[(A3)] $\|q_m(v_m\delta(x_{n(m,k)})v_m^*-v_mxv_m^*)\|_{\mathcal{M}}\leq
(3(k+1)2^{m})^{-1}$;
\item[(A4)] $\mathcal{D}(z_m(\mathbf{1}-q_m))\stackrel{(D6)}\leq \sum_{i=1}^3 \mathcal{D}(z_m(\mathbf{1}-q^{(i)}_{m,n(m,k)})) \leq 2^{-m-k-1}\varphi(z_m)$;
\item[(A5)] $1-\int_\Omega\varphi(z_m)d\mu=\int_\Omega\varphi(\mathbf{1}-z_m)d\mu\leq$
\begin{center}
   $\leq\sum_{i=1}^3\int_ \Omega\varphi(\mathbf{1}-z^{(i)}_{m,n(m,k)})d\mu\leq
2^{-m-k-1}.$
\end{center}
\end{itemize}

Fix $m_1,m_2\in\mathbb{N}$ with $m_1<m_2$ and set
$$q_{m_1,m_2}:=\inf_{m_1<m\leq m_2}q_m,\
z_{m_1,m_2}:=\inf_{m_1<m\leq m_2}z_m.$$

 Since $(\mathbf{1}-z_{m_1,m_2})=\sup\limits_{m_1<m\leq m_2}(\mathbf{1}-z_m)$ and $(\mathbf{1}-q_{m_1,m_2})=\sup\limits_{m_1<m\leq m_2}(\mathbf{1}-q_m)$, it follows that $\varphi(\mathbf{1}-z_{m_1,m_2})=\sup\limits_{m_1<m\leq m_2}\varphi(\mathbf{1}-z_m)$ and $\varphi(\mathbf{1}-q_{m_1,m_2})=\sup\limits_{m_1<m\leq m_2}\varphi(\mathbf{1}-q_m)$, and therefore
\begin{equation}
\label{m1m21}
1-\int_\Omega\varphi(z_{m_1,m_2})d\mu=\int_\Omega\varphi(\mathbf{1}-z_{m_1,m_2})d\mu\leq
\sum_{m=m_1+1}^{m_2}\int_\Omega\varphi(\mathbf{1}-z_m)d\mu\stackrel{(A5)}
\leq 2^{-m_1-k-1};
\end{equation}
\begin{equation}
\label{m1m22}
\mathcal{D}(z_{m_1,m_2}(\mathbf{1}-q_{m_1,m_2}))\stackrel{(D6)}\leq\sum_{m=m_1+1}^{m_2}\mathcal{D}(z_{m_1,m_2}(\mathbf{1}-q_m))
\stackrel{(A4)}\leq 2^{-m_1-k-1}\varphi(z_{m_1,m_2})
\end{equation}
and
\begin{equation}
\label{m1m23}
\|\sum_{m=m_1+1}^{m_2}(v_mx_{n(m,k)}v_m^*)q_{m_1,m_2}\|_{\mathcal{M}}\leq
\sum_{m=m_1+1}^{m_2}\|v_mx_{n(m,k)}v_m^*q_m\|_{\mathcal{M}}\stackrel{(A1)}\leq
2^{m_1}(k+1)^{-1}.
\end{equation}

Inequalities (\ref{m1m21})-(\ref{m1m23}) mean that the sequence $$S_{l,k}=\sum_{m=1}^lv_mx_{n(m,k)}v_m^*,\  l\geq 1$$ is a Cauchy sequence in the $F$-space $(LS(\mathcal{M},t(\mathcal{M})))$ for each fixed $k\in\mathbb{N}$. Consequently, there exists $y_k\in LS(\mathcal{M})$ such that $S_{l,k}\stackrel{t(\mathcal{M})}\longrightarrow y_k$ for
$l\longrightarrow\infty$,  i.e. the series
\begin{gather}\label{y_k}
y_k=\sum_{m=1}^\infty v_mx_{n(m,k)}v_m^*
\end{gather} converges in $LS(\mathcal{M})$ with respect to the topology $t(\mathcal{M})$. Since the involution is continuous in topology $t(\mathcal{M})$ and $S^*_{l,k}=S_{l,k}$, we conclude $y_k=y_k^*$.

Setting
\begin{equation}
\label{r_m}
r_m:=p_m\wedge q_m, \ m\in \mathbb{N},
\end{equation}
and using the relation $z_m(p_m-p_m\wedge q_m)\sim z_m(p_m\vee q_m-q_m)$ ( see e.g. \cite[ch. 5, Proposition 1.6]{Tak})
we have
\begin{equation}
\label{Dp_mq_m}
\begin{split}
\mathcal{D}(z_m(p_m-r_m))&=\mathcal{D}(z_m(p_m-p_m\wedge
q_m))\stackrel{(D3)}=\mathcal{D}(z_m(p_m\vee q_m - q_m))\cr \\&\leq
\mathcal{D}(z_m(\mathbf{1}-q_m))\stackrel{(A4)}\leq 2^{-m-k-1}\varphi(z_m).
\end{split}
\end{equation}
Setting
\begin{equation}
\label{q0z0}
q_0^{(k)}:=\sup_{m\geq 1} r_m,\ z_0^{(k)}:=\inf_{m\geq 1}z_m,
\end{equation}
we have (see (\ref{p_0}), (\ref{yp_0}) and (\ref{r_m}))
\begin{equation}
\label{p0q0}
 y \geq p_0 \geq q_0^{(k)},\ k\in\mathbb{N}.
\end{equation}

From (\ref{m1m21}) it follows that
\begin{equation}
\label{z0} 1-\int_\Omega\varphi(z_0^{(k)})d\mu=\int_\Omega\varphi(\mathbf{1}-z_0^{(k)})d\mu\leq 2^{-k-1}.
\end{equation}

Since $p_mp_j=0,\ m\neq j$, and $r_m\leq p_m$ (see (\ref{r_m})) we obtain $p_0-q_0^{(k)}=\sup_{m\geq 1}(p_m-r_m)$ and hence, by (\ref{Dp_mq_m}),
\begin{equation}
\label{Dz_kp0q0k} \mathcal{D}(z_0^{(k)}(p_0-q_0^{(k)}))\stackrel{(D6)}=\sum_{m=1}^\infty
\mathcal{D}(z_0^{(k)}(p_m-r_m))\stackrel{(\ref{Dp_mq_m})}\leq 2^{-k-1}\varphi(z_0^{(k)}).
\end{equation}

Due to (\ref{r_m}), we have $p_mq^{(k)}_0=r_mq^{(k)}_0=r_m=r_mq_m$ for all $m\in \mathbb{N}$. Hence,
$$v_mx_{n(m,k)}v_m^* q_0^{(k)}= v_mx_{n(m,k)}v_m^* p_mq_0^{(k)}= v_mx_{n(m,k)}v_m^* r_m$$
and
\begin{gather}
 \label{yq}\|y_k   q_0^{(k)}\|_{\mathcal{M}} =
\|(\sum_{m=1}^\infty v_mx_{n(m,k)}v_m^*)q_0^{(k)}\|_{\mathcal{M}}=
\notag\\
=
\|\sum_{m=1}^\infty v_mx_{n(m,k)}v_m^*q_0^{(k)}\|_{\mathcal{M}}\leq
\sup_{m\geq 1}\|
v_mx_{n(m,k)}v_m^* r_m\|_{\mathcal{M}}\leq
\notag\\
\leq
\sup_{m\geq 1}\|
v_mx_{n(m,k)}v_m^* q_m\|_{\mathcal{M}}\stackrel{(A1)}\leq (k+1)^{-1}.
\end{gather}

Using the properties of the derivation $\delta$ and equalities $p_nv_n=v_n,\ v_n^*=v_n^*p_n$ and (\ref{r_m}), (\ref{q0z0}), we have
\begin{equation*}
\begin{split}
q_0^{(k)} & \delta(v_mx_{n(m,k)}v_m^*)q_0^{(k)}\cr &=
q_0^{(k)}((\delta(v_mx_{n(m,k)}v_m^*)-v_mxv_m^*)+v_mxv_m^*)q_0^{(k)}\cr & =
(q_0^{(k)}\delta(v_m)x_{n(m,k)}v_m^*q_0^{(k)}+q_0^{(k)}v_mx_{n(m,k)}\delta(v_m^*)q_0^{(k)})\cr & +
q_0^{(k)}(v_m\delta(x_{n(m,k)})v_m^*-v_mxv_m^*)q_0^{(k)}+q_0^{(k)}(v_mxv_m^*)q_0^{(k)}\cr &
=
q_0^{(k)}\delta(v_m)x_{n(m,k)}v_m^*q_mr_m+r_mq_mv_mx_{n(m,k)}\delta(v_m^*)q_0^{(k)}\cr &+
r_mq_m(v_m\delta(x_{n(m,k)})v_m^*-v_mxv_m^*)q_mr_m+q_0^{(k)}(v_mxv_m^*)q_0^{(k)}.
\end{split}
\end{equation*}

Consider the following formal series suggested by the preceding

\begin{equation}
\label{s1}
\sum_{m=1}^\infty q_0^{(k)}\delta(v_m)x_{n(m,k)}v_m^*q_mr_m;
\end{equation}
\begin{equation}
\label{s2}
\sum_{m=1}^\infty r_mq_mv_mx_{n(m,k)}\delta(v_m^*)q_0^{(k)};
\end{equation}
\begin{equation}
\label{s3}
\sum_{m=1}^\infty r_mq_m(v_m\delta(x_{n(m,k)})v_m^*-v_mxv_m^*)q_mr_m;
\end{equation}
\begin{equation}
\label{s4}
\sum_{m=1}^\infty q_0^{(k)}(v_mx_{n(m,k)}v_m^*)q_0^{(k)}.
\end{equation}

By the condition (A2) the first series (\ref{s1}) and the second series (\ref{s2}) converge with respect to
the norm $\|.\|_{\mathcal{M}}$ to some elements
$a,b\in\mathcal{M}$ respectively and $\|a\|_{\mathcal{M}}\leq
(3(k+1))^{-1}$ and $\|b\|_{\mathcal{M}}\leq (3(k+1))^{-1}$.
Similarly, by the condition  (A3), the third series (\ref{s3})
also converges with respect to the norm
 $\|.\|_{\mathcal{M}}$ to some element $c\in\mathcal{M}$, satisfying $\|c\|_{\mathcal{M}}\leq
(3(k+1))^{-1}$. Finally, since $y=\sum_{m=1}^\infty v_mxv_m^*$
(the convergence of the latter series is taken in the $\tau_{so}$
topology), we see that the fourth series (\ref{s4}) converges with respect to the
topology $\tau_{so}$ to some element $q_0^{(k)}yq_0^{(k)}$. Hence, the
series
\begin{equation}\label{sq0deltaq0}
\sum_{m=1}^\infty q_0^{(k)}\delta(v_mx_{n(m,k)}v_m^*)q_0^{(k)}
\end{equation}
converges with respect to the
topology $\tau_{so}$ to some element $a_k\in\mathcal{M}$, and, in
addition, we have
\begin{equation}
\label{aq} \|a_k-q_0^{(k)}yq_0^{(k)}\|_{\mathcal{M}}\leq
(k+1)^{-1}.
\end{equation}

We shall show that
\begin{equation}
\label{qd} a_k=q_0^{(k)}\delta(y_k)q_0^{(k)},
\end{equation}
where $y_k=\sum_{m=1}^\infty v_mx_{n(m,k)}v_m^*$ (the convergence of the latter series is taken in the $t(\mathcal{M})$-topology (see (\ref{y_k})). Using (\ref{q0z0}) for any $m_1,m_2\in\mathbb{N}$
we have
\begin{equation*}
 \begin{split}
r_{m_1}q_0^{(k)}\delta(y_k)q_0^{(k)}r_{m_2}&=
\delta(r_{m_1}q_0^{(k)}y_k)q_0^{(k)}r_{m_2}-\delta(r_{m_1}q_0^{(k)})y_kq_0^{(k)}r_{m_2}\cr &
=\delta(r_{m_1}v_{m_1}x_{n(m_1,k)}v_{m_1}^*)r_{m_2}-\delta(r_{m_1})v_{m_2}x_{n(m_2,k)}v_{m_2}^*r_{m_2}.
\end{split}
\end{equation*}
Since the series $\sum_{m=1}^\infty q_0^{(k)}\delta(v_mx_{n(m,k)}v_m^*)q_0^{(k)}$ converges with respect to the topology $\tau_{so}$ (see \ref{sq0deltaq0}), it follows that the series $$\sum_{m=1}^\infty r_{m_1}(q_0^{(k)}\delta(v_mx_{n(m,k)}v_m^*)q_0^{(k)})r_{m_2}$$ also converges with respect to this topology (\cite[ch. VI]{R-S}), in addition, the following equalities hold
\begin{equation*}
 \begin{split}
r_{m_1}a_kr_{m_2}&=\sum_{m=1}^\infty r_{m_1}(q_0^{(k)}\delta(v_mx_{n(m,k)}v_m^*)q_0^{(k)})r_{m_2}\cr &\stackrel{(\ref{q0z0})}=\sum_{m=1}^\infty r_{m_1}\delta(v_mx_{n(m,k)}v_m^*)r_{m_2}\cr &=
 \sum_{m=1}^\infty (\delta(r_{m_1}v_mx_{n(m,k)}v_m^*)r_{m_2}-\delta(r_{m_1})v_mx_{n(m,k)}v_m^*r_{m_2})\cr &\stackrel{(\ref{r_m})}=
\delta(r_{m_1}v_{m_1}x_{n(m_1,k)}v_{m_1}^*)r_{m_2}-\delta(r_{m_1})v_{m_2}x_{n(m_2,k)}v_{m_2}^*r_{m_2},
\end{split}
\end{equation*}
which guarantees
\begin{equation}
\label{qk}
r_{m_1}q_0^{(k)}\delta(y_k)q_0^{(k)}r_{m_2}=r_{m_1}a_kr_{m_2}.
\end{equation}

Since $$r_{m_1}(\delta(y_k)-a_k)r_m\stackrel{(\ref{qk})}=0,$$ we see that for the right support $r(r_{m_1}(\delta(y_k)-a_k))$ of the operator $r_{m_1}(\delta(y_k)-a_k)$ satisfies the inequality   $$r(r_{m_1}(\delta(y_k)-a_k))\leq \mathbf{1}-r_m, m\in\mathbb{N},$$ and therefore $$r(r_{m_1}(\delta(y_k)-a_k))\leq \inf_{m\geq 1}(\mathbf{1}-r_m)\stackrel{(\ref{q0z0})}=\mathbf{1}-q_0^{(k)}.$$ Consequently, $r_{m_1}(\delta(y_k)-a_k)q_0^{(k)}=0$ for all $m_1\in\mathbb{N}$.

Similarly, using the left support of the operator $(\delta(y_k)-a_k)q_0^{(k)}$, we claim that $q_0^{(k)}(\delta(y_k)-a_k)q_0^{(k)}=0$.

Since $q_0^{(k)}a_kq_0^{(k)}=a_k$, the equality  (\ref{qd}) holds.

Thus, the inequality (\ref{aq}) can be restated as follows
\begin{equation}
\label{aqk}
\|q_0^{(k)}(\delta(y_k)-y)q_0^{(k)}\|_{\mathcal{M}}\leq
(k+1)^{-1}.
\end{equation}
It follows from the inequalities (\ref{yq}) and (\ref{aqk}), that
\begin{equation}
\label{ky} \|(k+1)q_0^{(k)}y_k\|_{\mathcal{M}}=\|(k+1)y_kq_0^{(k)}\|_{\mathcal{M}}\leq 1
\end{equation}
and
\begin{equation}
\label{yky}
\|q_0^{(k)}\delta((k+1)y_k)q_0^{(k)}-(k+1)q_0^{(k)}yq_0^{(k)}\|_{\mathcal{M}}\leq
1.
\end{equation}
Due to (\ref{yky}), and taking into account (\ref{p0q0}), we obtain
\begin{equation*}
 \begin{split}
(k+1)q_0^{(k)}-q_0^{(k)}\delta((k+1)y_k)q_0^{(k)}&\leq
(k+1)q_0^{(k)}yq_0^{(k)}-q_0^{(k)}\delta((k+1)y_k)q_0^{(k)}\cr &\leq
q_0^{(k)},
\end{split}
\end{equation*} that is
\begin{equation}
\label{yqky} kq_0^{(k)}\leq q_0^{(k)}\delta((k+1)y_k)q_0^{(k)}.
\end{equation}

Let us now consider the projections
\begin{equation}
\label{q0z01}q_0:=\inf_{k\geq 1} q_0^{(k)},\ z_0:=\inf_{k\geq 1} z_0^{(k)}.
\end{equation}

Using (\ref{p0q0}), (\ref{q0z01}) we have that $p_0-q_0=\sup_{k\geq 1} (p_0-q_0^{(k)})$.  Therefore, combining (\ref{Dz_kp0q0k}) and (\ref{q0z01}), we obtain
\begin{equation}
\label{Dq0z01}\mathcal{D}(z_0(p_0-q_0))=\mathcal{D}(\sup\limits_{k\geq 1}(z_0(p_0-q_0^{(k)})))\stackrel{(D6)}
\leq \sum_{k=1}^\infty
\mathcal{D}(z_0(p_0-q_0^{(k)}))\stackrel{(\ref{Dz_kp0q0k})}\leq \varphi(z_0),
\end{equation}
 that is the projection
$z_0(p_0-q_0)$ is finite (see (D1)). Moreover, due to inequalities
(\ref{ky}) (respectively, (\ref{yqky})), we have
\begin{equation}
\label{kyq}\|(k+1)q_0y_k\|_{\mathcal{M}}=\|(k+1)y_kq_0\|_{\mathcal{M}}\leq 1,\quad k\in\mathbb{N}
\end{equation}
(respectively,
\begin{equation}
\label{kqy}kq_0\leq q_0\delta((k+1)y_k)q_0,\quad k\in\mathbb{N}.)
\end{equation}

 Since $\varphi$ is a $*$-isomorphism from $\mathcal{Z(M)}$ onto $L^\infty(\Omega,\Sigma,\mu)$, by (\ref{z0}), we have that
$$\int_\Omega \varphi(\mathbf{1}-z_0)d\mu=\int_\Omega \sup_{k\geq 1}\varphi(\mathbf{1}-z_0^{(k)})d\mu\leq\sum_{k=1}^\infty \int_\Omega \varphi(\mathbf{1}-z_0^{(k)})d\mu \stackrel{(\ref{z0})}\leq 2^{-1},$$ in particular, $z_0\neq 0$.
 Since $\mathbf{1}=c(p_0)$ and $c(p_0z_0)=c(p_0)z_0=z_0\neq 0$, we have $z_0p_0\neq 0$, and therefore there exists such $n\in\mathbb{N}$ that $z_0p_n\neq 0$ (see (\ref{p_0})).
Since $z_0p_n\sim z_0p_m$, we have $z_0p_m\neq 0$ for all
$m\in\mathbb{N}$. Hence, $z_0p_0$ is an infinite projection.
Since the projection $z_0(p_0-q_0)$ is finite (see (\ref{Dq0z01})), we see that the projection $z_0q_0$ must be infinite. By \cite[Proposition 6.3.7]{KR},
there exists a central projection $$0\neq
e_0\in\mathcal{P}(\mathcal{Z}(\mathcal{M})),\ e_0\leq
z_0,$$ such that $e_0q_0$ is properly infinite,
in particular, there exist pairwise orthogonal projections
\begin{equation}
\label{e0}e_n\leq e_0q_0,\ e_n\sim e_0q_0
\end{equation}
for all
$n\in\mathbb{N}$ (see, for example, \cite[Proposition 2.2.4]{Sak}).  In addition,
\begin{equation}\label{int_e0q0}
\int_\Omega\varphi(c(q_0)e_0)\,d\mu\neq 0.
\end{equation}

For every $n\in\mathbb{N}$ the operator
$$b_n:=\delta(e_n)e_n$$ is locally measurable, and therefore there exists such a sequence
$\{z_m^{(n)}\} \subset\mathcal{P}(\mathcal{Z}(\mathcal{M}))$ that
$z_m^{(n)}\uparrow \mathbf{1}$ when $m\rightarrow\infty$ and $z_m^{(n)}b_n\in S(\mathcal{M})$ for all $m\in\mathbb{N}$.
 Since $\varphi(z_m^{(n)})\uparrow \varphi(\mathbf{1})=\mathbf{1}_{L^\infty(\Omega)}$ it follows that   $\int_\Omega\varphi(z_m^{(n)})d\mu \uparrow \mu(\mathbf{1}_{L^\infty(\Omega)})=1$ when
$m\rightarrow\infty$, and therefore, by (\ref{int_e0q0}), for every $n\in\mathbb{N}$ there exists such a projection  $z^{(n)}\in\mathcal{P}(\mathcal{Z}(\mathcal{M}))$,
that $z^{(n)}b_n\in S(\mathcal{M})$ and
\begin{equation}
\label{int}1-2^{-n-1}\int_\Omega\varphi(c(q_0)e_0)d\mu<\int_\Omega\varphi(z^{(n)})d\mu.
\end{equation}
Consider the central projection
$$g_0:=\inf_{n\geq 1}z^{(n)}.$$
  Since $z^{(n)}b_n\in S(\mathcal{M}), g_0=g_0z^{(n)}$ we have that 
$g_0b_n\in S(\mathcal{M})$ for all $n\in\mathbb{N}$. Due to (\ref{int}) we have
\begin{gather*}
\begin{split}1&-\int_\Omega\varphi(g_0)d\mu=\int_\Omega\varphi(\mathbf{1}-g_0)d\mu=\int_\Omega\sup\varphi(\mathbf{1}-z^{(n)})d\mu\leq\\& \sum_{n=1}^\infty\int_\Omega\varphi(\mathbf{1}-z^{(n)})d\mu=\sum_{n=1}^\infty(1-\int_\Omega\varphi(z^{(n)})d\mu)\leq 2^{-1}\int_\Omega\varphi(c(q_0)e_0)d\mu.
\end{split}
\end{gather*}

Consequently, $1-2^{-1}\int_\Omega\varphi(c(q_0)e_0)d\mu\leq \int_\Omega\varphi(g_0)d\mu$, and therefore
\begin{equation}
\label{phi}1+2^{-1}\int_\Omega\varphi(c(q_0)e_0)d\mu\leq \int_\Omega\varphi(g_0)d\mu+\int_\Omega\varphi(c(q_0)e_0)d\mu.
\end{equation}
 From (\ref{int_e0q0}) and inequality (\ref{phi}), it follows that $\int_\Omega \varphi(g_0c(q_0)e_0)d\mu>0$, i.e. $g_0c(q_0)e_0\neq 0$
and so $g_0e_0q_0\neq 0$. Since $e_0q_0$ is a properly infinite projection it follows that $g_0e_0q_0$ is a properly infinite projection. From the relationship $g_0e_n\stackrel{(\ref{e0})}\sim g_0e_0q_0$,
we see that the projection $g_0e_n$ is
also properly infinite for all $n\in\mathbb{N}$.  Since
$$c(g_0e_n)=g_0c(e_n)\stackrel{(\ref{e0})}\leq q_0c(q_0e_0)=g_0c(q_0)e_0,$$ it follows that $ze_n$ is
also properly infinite projection  for every $0\neq z \in
\mathcal{P}(\mathcal{Z}(\mathcal{M}))$ with $z\leq
g_0c(q_0)e_0$. Indeed, if $z'\in\mathcal{P(Z(M))}$ and $z'ze_n\neq 0$, then $0\neq z'ze_n=(z'zc(q_0)e_0)g_0e_n$, and therefore, since the projection $g_0e_n$ is properly infinite, we have $(z'zc(q_0)e_0)g_0e_n\notin\mathcal{P}_{fin}(\mathcal{M})$. Consequently, the projection $ze_n$ is also properly infinite.

Passing, if necessary to the algebra
$g_0c(q_0)e_0\mathcal{M}$, we may assume that
$g_0c(q_0)e_0=\mathbf{1}$. In this case, we also may assume that
$b_n\in S(\mathcal{M}),\ e_n\sim q_0$, $c(e_n)=\mathbf{1}$ and $ze_n$ is a properly infinite projection for every non-zero
$z\in\mathcal{P}(\mathcal{Z}(\mathcal{M}))$.

The assumption $b_n\in S(\mathcal{M})$ means that for every fixed $n\in\mathbb{N}$ there exists such a sequence
$\{p^{(n)}_m\}_{m=1}^\infty\subset
\mathcal{P}_{fin}(\mathcal{M})$, that $p^{(n)}_m\downarrow 0$ when
$m\rightarrow\infty$ and $b_n(\mathbf{1}-p^{(n)}_m)\in\mathcal{M}$
for all $m\in\mathbb{N}$. Since $\mathcal{D}(p^{(n)}_m)\in
L^0(\Omega,\Sigma,\mu)$ and $\mathcal{D}(p^{(n)}_m)\downarrow 0$ (see (D7)),  it follows that $\{\mathcal{D}(p_m^{(n)})\}_{n=1}^\infty$ converges in measure $\mu$ to zero. Consequently,
we may select a central projection  $f_n$ and a finite projection
$s_n=p_{m_n}^{(n)}\in \mathcal{P}_{fin}(\mathcal{M})$ as to guarantee
$\mathcal{D}(f_ns_n)<2^{-n}\varphi(f_n)$,
$1-2^{-n-1}<\int\varphi(f_n)d\mu$ and

\begin{equation}\label{f_nb_n}
 f_nb_n(\mathbf{1}-s_n)\in\mathcal{M}
\end{equation} for all
$n\in\mathbb{N}$.

Setting $$f:=\inf_{n\geq 1} f_n,\ s:=\sup_{n\geq 1}s_n,$$
we have that $$1/2<\int\varphi(f)d\mu,\quad
\mathcal{D}(fs)\stackrel{(D6)}\leq\sum_{n=1}^\infty\mathcal{D}(fs_n)\leq
\varphi(f).$$

This means that $f\neq 0$ and $fs\in
\mathcal{P}_{fin}(\mathcal{M})$ (see (D1)). In addition,  since $f\leq f_n, (\mathbf{1}-s)\leq (\mathbf{1}-s_n)$ from (\ref{f_nb_n}) it follows that 
$fb_n(\mathbf{1}-s)\in\mathcal{M}$ for all $n\in\mathbb{N}$.

Consider the projections $t=f(\mathbf{1}-s)$ and $g_n=f(e_n\wedge(\mathbf{1}-s)),\
n\in\mathbb{N}$. Clearly (see (\ref{e0})),
\begin{equation}
\label{gze} g_n\leq fe_n\leq q_0,\ b_ng_n\in\mathcal{M},\ g_n\leq t
\end{equation}
for all $n\in\mathbb{N}$, and also
$$fe_n-g_n=f(e_n-e_n\wedge(\mathbf{1}-s))\sim
f(e_n\vee(\mathbf{1}-s)-(\mathbf{1}-s))\leq fs,$$ that is
$fe_n-g_n\in \mathcal{P}_{fin}(\mathcal{M})$. Hence,
for every non-zero central projection  $z\leq f$, we have that the projection
 $ze_n-zg_n$ is finite. Since the projection $ze_n$
is infinite, the projection $zg_n$ is also infinite, i.e.
\begin{equation}
\label{zgn} zg_n\notin\mathcal{P}_{fin}(\mathcal{M})
\end{equation}
for any $0\neq z\in \mathcal{P}(\mathcal{Z}(\mathcal{M}))$ and
$n\in\mathbb{N}$.

Since $b_nt=fb_n(\mathbf{1}-s) \in \mathcal{M}$, we see that there exists such an increasing sequence
 $\{l_n\}\subset\mathbb{N}$ that
$l_n>n+2\|b_nt\|_{\mathcal{M}}$ for all $n\in\mathbb{N}$.
Appealing to the inequalities (\ref{kyq}), (\ref{gze}) and taking into account the equality $b_n=\delta(e_n)e_n$, we deduce
\begin{equation*}
\begin{split}
\|g_n({l_n}+1)y_{l_n}\delta(e_n)e_ng_n\|_{\mathcal{M}}&\leq
\|g_n({l_n}+1)y_{l_n}\|_{\mathcal{M}}\|\delta(e_n)e_ng_n\|_{\mathcal{M}}\cr &\leq
\|q_0({l_n}+1)y_{l_n}\|_{\mathcal{M}}\|\delta(e_n)e_nt\|_{\mathcal{M}}\cr &<
(l_n-n)/2.
\end{split}
\end{equation*}
Hence,
\begin{equation}\label{ineq_l_n-n}
\|g_ne_n\delta(e_n)({l_n}+1)y_{l_n}g_n+g_n({l_n}+1)y_{l_n}\delta(e_n)e_ng_n\|_{\mathcal{M}}\leq
l_n-n.
\end{equation}

 For every $x=x^*\in\mathcal{M}$ the inequalities $-\|x\|_\mathcal{M}\mathbf{1}\leq x\leq \|x\|_\mathcal{M}\mathbf{1}$ holds, in particular, $-g_n\|x\|_\mathcal{M}\leq q_nxq_n\leq g_n\|x\|_\mathcal{M}$. Hence, inequality (\ref{ineq_l_n-n}) implies that
\begin{equation}
\label{ge}g_ne_n\delta(e_n)({l_n}+1)y_{l_n}g_n+g_n({l_n}+1)y_{l_n}\delta(e_n)e_ng_n\geq
(n-l_n)g_n.
\end{equation}

Since $e_ne_m=0$ whenever $n\neq m$, we see (due to inequalities
(\ref{kyq}) and (\ref{gze})) that the series $\sum_{n=1}^\infty
e_n({l_n}+1)y_{l_n}e_n$ converges with respect to the topology $\tau_{so}$ to a
self-adjoint operator $h_0\in\mathcal{M}$, satisfying
$$\|h_0\|_{\mathcal{M}}\leq \sup_{n\geq
1}\|e_n({l_n}+1)y_{l_n}e_n\|_{\mathcal{M}}\leq 1.$$ Again
appealing to the inequalities (\ref{kqy}), (\ref{gze}) and
(\ref{ge}), we infer that
\begin{equation*}
\begin{split}
n g_n & =l_ng_n+(n-l_n)g_n  \cr &\leq
g_n({l_n}+1)\delta(y_{l_n})g_n+g_ne_n\delta(e_n)({l_n}+1)y_{l_n}g_n+g_n({l_n}+1)y_{l_n}\delta(e_n)e_ng_n\cr &=
({l_n}+1)(g_n\delta(y_{l_n})g_n+g_ne_n\delta(e_n)y_{l_n}g_n+g_ny_{l_n}\delta(e_n)e_ng_n)\cr &=
({l_n}+1)g_n\delta(e_ny_{l_n}e_n)g_n\cr &=
\delta(g_ne_n({l_n}+1)y_{l_n}e_n)g_n-\delta(g_n)e_n({l_n}+1)y_{l_n}e_ng_n\cr &=
\delta(g_nh_0)g_n-\delta(g_n)h_0g_n= g_n\delta(h_0)g_n.
\end{split}
\end{equation*}
 Thus,
\begin{equation}\label{ng_n}
n g_n\leq g_n\delta(h_0)g_n
\end{equation}  for every
$n\in\mathbb{N}$.

Set $g_n^{(0)}=g_n\wedge E_{n-1}(\delta(h_0)),\ n\in\mathbb{N}$, where $\{E_\lambda(\delta(h_0))\}$ is the spectral family of projections for self-adjoint operator $\delta(h_0)$. For every $n\in\mathbb{N}$ we have
\begin{gather*}
\begin{split}
ng_n^{(0)} & =ng_n^{(0)}g_ng_n^{(0)}\stackrel{(\ref{ng_n})}\leq g_n^{(0)}(g_n\delta(h_0)g_n)g_n^{(0)}\cr &=
g_n^{(0)}\delta(h_0)g_n^{(0)}=g_n^{(0)}E_{n-1}(\delta(h_0))\delta(h_0)g_n^{(0)}\cr & \leq
g_n^{(0)}(n-1)E_{n-1}(\delta(h_0))g_n^{(0)}=(n-1)g_n^{(0)}.
\end{split}
\end{gather*}
Hence, $g_n\wedge E_{n-1}(\delta(h_0))=g_n^{(0)}=0$ which implies $$g_n=g_n-g_n\wedge E_{n-1}(\delta(h_0))\sim g_n\vee E_{n-1}(\delta(h_0))-E_{n-1}(\delta(h_0))\leq \mathbf{1}-E_{n-1}(\delta(h_0)),$$ i.e. $g_n\preceq \mathbf{1}-E_{n-1}(\delta(h_0))$.

Then $g_n\stackrel{(\ref{gze})}\leq fg_n\preceq f(\mathbf{1}-E_{n-1}(\delta(h_0)))$, and therefore
\begin{equation}
\label{Dgn} \mathcal{D}(g_n)\stackrel{(D3)}\leq \mathcal{D}(f(\mathbf{1}-E_{n-1}(\delta(h_0))))
\end{equation}
for all $n\in\mathbb{N}$.

Since $|f\delta(h_0)|\in LS(\mathcal{M})$, we see that there exists
such a non-zero central projection $f_0\leq f$,
that $|f_0\delta(h_0)|\in S_h(\mathcal{M})$.
Hence, we may find such $\lambda_0>0$, that
$(f_0-E_{\lambda}(|f_0\delta(h_0)|))\in
\mathcal{P}_{fin}(\mathcal{M})$ for all $\lambda\geq \lambda_0$
(\cite[\S2.2]{MCh}), that is
$\mathcal{D}(f_0(\mathbf{1}-E_{\lambda}(|f_0\delta(h_0)|)))\in
L_+^0(\Omega,\Sigma,\mu)$ when $\lambda>\lambda_0$.

Since
$f_0(\mathbf{1}-E_{\lambda}(|f_0\delta(h_0)|))=
f_0(\mathbf{1}-E_{\lambda}(|\delta(h_0)|))$, we infer from (\ref{Dgn}) that
$$\mathcal{D}(f_0g_n)\in L_+^0(\Omega,\Sigma,\mu)$$ for all $n\geq
\lambda_0+1$ which contradicts with the property (D1) in the
definition of the dimension function $\mathcal{D}$, since
$f_0g_n$ is an infinite projection (see (\ref{zgn})).

Hence, our assumption that the derivation
$\delta$ fails to be continuous in
$(LS(\mathcal{M}),t(\mathcal{M}))$ has led to a contradiction.
\end{proof}

Observe that in the special case of properly infinite von Neumann algebras of type
 $I$ or $III$ , Theorem \ref{main} gives a new proof of the results concerning the continuity of a derivation of $(LS(\mathcal{M}),t(\mathcal{M}))$ established earlier in  \cite{AAK,AK,BdPS}.

\section{Extension of a derivation $\delta : \mathcal{M}\rightarrow LS(\mathcal{M})$ up to a derivation on $LS(\mathcal{M})$}

In this section the construction of extension of any derivation, acting on a von Neumann algebra $\mathcal{M}$ with values in $LS(\mathcal{M})$, up to a derivation from $LS(\mathcal{M})$ into $LS(\mathcal{M})$ is given. Using this extension and Theorem \ref{main} it is established that in case the of a properly infinite von Neumann algebra $\mathcal{M}$, any derivation $\delta: \mathcal{A}\longrightarrow LS(\mathcal{M})$ from a subalgebra $\mathcal{A}$ satisfying $\mathcal{M}\subset\mathcal{A}\subset LS(\mathcal{M})$ is continuous with respect to the local measure topology.

Let $\mathcal{\mathcal{M}}$ be an arbitrary von Neumann algebra and let $\{z_n\}_{n=1}^\infty$ be a sequence of central projections from $\mathcal{\mathcal{M}}$, such that $z_n\uparrow \mathbf{1}$. A sequence $\{x_n\}_{n=1}^\infty$ is called \emph{consistent} with the sequence $\{z_n\}_{n=1}^\infty$, if for any $n,m\in\mathbb{\mathbb{N}}$ the equality $x_mz_n=x_nz_n$ holds for $n<m$.

\begin{proposition}
\label{p8} Let $\{x_n\}_{n=1}^\infty\subset LS(\mathcal{M})$ (respectively, $\{y_n\}_{n=1}^\infty\subset LS(\mathcal{M})$) be a sequence consistent with the sequence $\{z_n\}_{n=1}^\infty\subset\mathcal{P}(\mathcal{Z}(\mathcal{M}))$ (respectively, with the sequence $\{z'_n\}_{n=1}^\infty\subset\mathcal{P}(\mathcal{Z}(\mathcal{M}))$), $z_n\uparrow \mathbf{1}$ ($z'_n\uparrow \mathbf{1}$). Then

(i). There exists a unique $x\in LS(\mathcal{M})$, such that $xz_n=x_nz_n$ for all $n\in\mathbb{N}$, in addition, $x_n\stackrel{t(\mathcal{M})}{\longrightarrow} x$;

(ii). If $x_nz_nz'_m=y_mz_nz'_m$ for all $n,m\in\mathbb{N}$, then $(x_nz_n-y_nz'_n)\stackrel{t(\mathcal{M})}{\longrightarrow} 0$ for $n\rightarrow\infty$.
\end{proposition}
\begin{proof}
(i). Consider a neighborhood $V(B,\varepsilon,\delta)$ of zero in topology $t(\mathcal{M})$, where $\varepsilon,\delta>0,\ B\in\Sigma,\ \mu(B)<\infty$ (see the definition of topology $t(\mathcal{M})$ in section 2). Since $z_n^\bot=(\mathbf{1}-z_n)\downarrow 0$, it follows that $\varphi(z_n^\bot)\in W(B,\varepsilon,\delta)$ for $n\geq n(B,\varepsilon,\delta)$. Taking $x\in LS(\mathcal{M}),\  q_n=z_n$, we have $(xz_n^\bot) q_n=0,\ \mathcal{D}(z_n^\bot q_n)=0$, i.e. $xz_n^\bot\in V(B,\varepsilon,\delta)$ for all $x\in LS(\mathcal{M}),\ n\geq n(B,\varepsilon,\delta)$. For $m>n$, we have  $$x_mz_m-x_nz_n=x_mz_m-x_mz_n=x_m(z_m-z_n)=x_mz_mz_n^\bot\in V(B,\varepsilon,\delta)$$ for all $n\geq n(B,\varepsilon,\delta)$. It means that $\{x_nz_n\}_{n=1}^\infty$ is a Cauchy sequence in $(LS(\mathcal{M}),t(\mathcal{M}))$. Consequently, there exists $x\in LS(\mathcal{M})$ such that $x_nz_n\stackrel{t(\mathcal{M})}{\longrightarrow} x$.

Since $x_nz_n^\bot\in V(B,\varepsilon,\delta)$ for all $n\geq n(B,\varepsilon,\delta)$, it follows that $x_nz_n^\bot\stackrel{t(\mathcal{M})}{\longrightarrow} 0$, and therefore $x_n=x_nz_n+x_nz_n^\bot\stackrel{t(\mathcal{M})}{\longrightarrow} x$. Fixing $k\in\mathbb{N}$, for $n>k$ we have $x_kz_k=x_nz_k\stackrel{t(\mathcal{M})}{\longrightarrow} xz_k$ for $n\rightarrow\infty$, i.e. $xz_k=x_kz_k$ for all $k\in\mathbb{N}$.

If $a\in LS(\mathcal{M})$ and $az_n=x_nz_n=xz_n$ for all
$n\in\mathbb{N}$, then $0=(a-x)z_n\stackrel{t(\mathcal{M})}{\longrightarrow} (a-x)$, i.e.
$a=x$.

(ii). If $x_m z_m {z'}_{n}^{\bot} \stackrel{t(\mathcal{M})}{\longrightarrow} 0$ for $n\rightarrow\infty$, $y_nz'_nz_m^\bot\stackrel{t(\mathcal{M})}{\longrightarrow} 0$ for $m\rightarrow\infty$, and $x_nz_n-x_mz_m\stackrel{t(\mathcal{M})}{\longrightarrow} 0$ for $n,m\rightarrow\infty$, then
\begin{gather*}
\begin{split}
x_nz_n-y_nz'_n&=x_nz_n-x_mz_m+x_mz_mz'_n+x_mz_m{z'}_n^\bot-y_nz'_n=
\\
&=(x_nz_n-x_mz_m)+y_nz_mz'_n+x_mz_m{z'}_n^\bot-y_nz'_n=
\\
&=
(x_nz_n-x_mz_m)-y_nz'_nz_m^\bot+x_mz_m{z'}_n^\bot\stackrel{t(\mathcal{M})}{\longrightarrow} 0
\end{split}
\end{gather*}
for $n,m\rightarrow\infty$.
\end{proof}

Now, we consider a derivation $\delta$ from $S(\mathcal{M})$ into
$LS(\mathcal{M})$ and construct an extension $\widetilde{\delta}$ from $LS(\mathcal{M})$ into
$LS(\mathcal{M})$. Recall that for an arbitrary operator $x\in LS(\mathcal{M})$
there exists a sequence $\{z_n\}_{n=1}^\infty\subset
\mathcal{P}(\mathcal{Z}(\mathcal{M}))$ such that
$z_n\uparrow\mathbf{1}$ and $xz_n\in S(\mathcal{M})$ for all
$n\in\mathbb{N}$.

Since $\delta(xz_n)z_m=\delta(xz_nz_m)$ (see Lemma \ref{l1}), the
sequence $\{\delta(xz_n)\}_{n=1}^\infty$ is consistent with the
sequence $\{z_n\}_{n=1}^\infty$. By Proposition \ref{p8}(i), there
exists a unique $y(x)\in LS(\mathcal{M})$ such that
$\delta(xz_n)\stackrel{t(\mathcal{M})}{\longrightarrow} y(x)$
(notation:
$y(x)=t(\mathcal{M})-\lim_{n\rightarrow\infty}\delta(xz_n)$). Set
$\widetilde{\delta}(x)=y(x)$. According to Proposition
\ref{p8}(ii), the definition of operator $\widetilde{\delta}(x)$
does not depend on a choice of a sequence
$\{z_n\}_{n=1}^\infty\subset\mathcal{P}(\mathcal{Z}(\mathcal{M}))$, for which
$z_n\uparrow\mathbf{1}$ and $xz_n\in S(\mathcal{M})$,
$n\in\mathbb{N}$. If $x\in S(\mathcal{M})$, then, taking
$z_n=\mathbf{1},\ n\in\mathbb{N}$, we obtain
$\widetilde{\delta}(x)=\delta(x)$.

\begin{proposition}
\label{p9}
The mapping $\widetilde{\delta}$ is a unique derivation from $LS(\mathcal{M})$ into $LS(\mathcal{M})$ such that $\widetilde{\delta}(x)=\delta(x)$ for all $x\in S(\mathcal{M})$.
\end{proposition}
\begin{proof}
Let $x,y\in LS(\mathcal{M})$, and let $z_n,p_n\in\mathcal{P}(\mathcal{Z}(\mathcal{M}))$ be such that$z_n\uparrow\mathbf{1},\ p_n\uparrow\mathbf{1},\ xz_n,yp_n\in S(\mathcal{M}),\ n\in\mathbb{N}$. Observing that $$z_np_n\in\mathcal{P}(\mathcal{Z}(\mathcal{M})),\ (z_np_n)\uparrow\mathbf{1},\ xz_np_n,yz_np_n,(x+y)z_np_n\in S(\mathcal{M}),\ n\in\mathbb{N},$$ we have
\begin{gather*}
\begin{split}
\widetilde{\delta}(x+y)&=t(\mathcal{M})-\lim_{n\rightarrow\infty}\delta((x+y)z_np_n)=
 \\
&=\bigl(t(\mathcal{M})-\lim_{n\rightarrow\infty}\delta(xz_np_n)\bigl) +\bigl(t(\mathcal{M})-\lim_{n\rightarrow\infty}\delta(yz_np_n)\bigl)=
\\
&=\widetilde{\delta}(x)+\widetilde{\delta}(y).
\end{split}
\end{gather*}

Similarly, $\widetilde{\delta}(\lambda x)=\lambda\widetilde{\delta}(x)$, $\lambda\in\mathbb{C}$. Further, using convergences $$xz_n\stackrel{t(\mathcal{M})}{\longrightarrow} x,\ yp_n\stackrel{t(\mathcal{M})}{\longrightarrow} y,\ \delta(xz_n)\stackrel{t(\mathcal{M})}{\longrightarrow} \widetilde{\delta}(x),\ \delta(yp_n)\stackrel{t(\mathcal{M})}{\longrightarrow} \widetilde{\delta}(y)$$ and the inclusion $xyz_np_n\in S(\mathcal{M}),\ n\in\mathbb{N}$, we have
\begin{gather*}
\begin{split}
\widetilde{\delta}(xy)&= t(\mathcal{M})-\lim_{n\rightarrow\infty}\delta(xyz_np_n)=t(\mathcal{M})-\lim_{n\rightarrow\infty}\delta((xz_n)(yp_n))=\\
&=t(\mathcal{M})-\lim_{n\rightarrow\infty}(\delta(xz_n)yp_n+xz_n\delta(y_np_n))=\widetilde{\delta}(x)y+x\widetilde{\delta}(y).
\end{split}
\end{gather*}

Consequently, $\widetilde{\delta}: LS(\mathcal{M})\rightarrow LS(\mathcal{M})$ is a derivation, in addition, $\widetilde{\delta}(x)=\delta(x)$ for all $x\in S(\mathcal{M})$.

Assume that $\delta_1: LS(\mathcal{M})\rightarrow LS(\mathcal{M})$ is also a derivation for which $\delta_1(x)=\delta(x)$ for all $x\in S(\mathcal{M})$. Let us show that $\widetilde{\delta}=\delta_1$.

If $x\in LS(\mathcal{M}),\ z_n\uparrow\mathbf{1},\ xz_n\in S(\mathcal{M}),\ n\in\mathbb{N}$, then, by Lemma \ref{l1} and Proposition \ref{p8} (i), we obtain
\begin{gather*}
\begin{split}
\widetilde{\delta}(x)&=t(\mathcal{M})-\lim_{n\rightarrow\infty}\delta(xz_n)=t(\mathcal{M})-\lim_{n\rightarrow\infty}\delta_1(xz_n)=\\
&= t(\mathcal{M})-\lim_{n\rightarrow\infty}\delta_1(x)z_n=\delta_1(x).
\end{split}
\end{gather*}
\end{proof}

Now, we give the construction of extension of a derivation $\delta: \mathcal{M}\rightarrow LS(\mathcal{M})$ up to a derivation $\widehat{\delta}: S(\mathcal{M})\rightarrow LS(\mathcal{M})$. For each $x\in LS(\mathcal{M})$ set $s(x):=l(x)\vee r(x)$, where $l(x)$ is the left and $r(x)$ is the right support of $x$. If $x=u|x|$ is a polar decomposition of $x\in LS(\mathcal{N})$, then $u\in\mathcal{M}$ \cite[\S2.3]{MCh} and, due to equalities $l(x)=uu^*,\ r(x)=u^*u$, we have $l(x)\sim r(x)$. We need the following lemma.

\begin{lemma}
\label{l3}
If $\mathcal{D}$ is a dimension function of a von Neumann algebra $\mathcal{M}$, then for any derivation $\delta$ from $\mathcal{M}$ into $LS(\mathcal{M})$ the following inequality $$\mathcal{D}(s(\delta(x)))\leq 3 \mathcal{D}(s(x))$$ holds for all $x\in\mathcal{M}$.
\end{lemma}
\begin{proof}
For $x\in\mathcal{M}$ we have
\begin{gather*}
l(\delta(x)s(x))\sim r(\delta(x)s(x))\leq s(x),\\
r(x\delta(s(x)))\sim l(x\delta(s(x)))=l(s(x)x\delta(s(x)))\leq s(x),
\end{gather*}
i.e.
$$l(\delta(x)s(x))\preceq s(x)$$ and $$r(x\delta(s(x)))\preceq s(x),$$ that implies the inequalities  (see (D2), (D3))
$$\mathcal{D}(l(\delta(x)s(x)))\leq\mathcal{D}(s(x)),\ \mathcal{D}(r(x\delta(s(x))))\leq\mathcal{D}(s(x)).$$

Since $$\delta(x)=\delta(xs(x))=\delta(x)s(x)+x\delta(s(x)),$$ we have $$s(\delta(x))=s(\delta(x)s(x)+x\delta(s(x)))\leq s(x)\vee l(\delta(x)s(x))\vee r(x\delta(s(x))).$$ Due to (D6), we have $$\mathcal{D}(s(\delta(x)))\leq \mathcal{D}(s(x))+\mathcal{D}(l(\delta(x)s(x)))+\mathcal{D}(r(x\delta(s(x))))\leq 3 \mathcal{D}(s(x)).$$
\end{proof}

As in the definition of the topology $t(\mathcal{M})$, denote by $\varphi$ a $*$-isomorphism from $\mathcal{Z}(\mathcal{M})$ onto the $*$-algebra $L^\infty(\Omega,\Sigma,\mu)$, where $\mu$ is a measure satisfying the direct sum property. By Proposition \ref{plm-spk1}(i), the convergence of the sequence of projections $p_n\stackrel{t(\mathcal{M})}{\longrightarrow} 0$ is equivalent to existence of a sequence $\{z_n\}\subset\mathcal{P}(\mathcal{Z}(\mathcal{M}))$ such that $z_np_n\in \mathcal{P}_{fin}(\mathcal{M})$ for all $n$, $\varphi(z_n^\bot)\stackrel{t(L^\infty(\Omega))}{\longrightarrow} 0$ and $\mathcal{D}(z_np_n)\stackrel{t(L^\infty(\Omega))}{\longrightarrow} 0$.

\begin{lemma}
\label{l4}
If $\{x_n\}_{n=1}^\infty\subset LS(\mathcal{M}),\ s(x_n)\in\mathcal{P}_{fin}(\mathcal{M}),\ \mathcal{D}(s(x_n))\stackrel{t(L^\infty(\Omega))}{\longrightarrow} 0$, then $x_n\stackrel{t(\mathcal{M})}{\longrightarrow} 0$.
\end{lemma}
\begin{proof}
Taking $z_n=\mathbf{1}$ for all $n\in\mathbb{N}$, we have $$z_ns(x_n)\in\mathcal{P}_{fin}(\mathcal{M}),\ \varphi(z_n^\bot)=0,\ n\in\mathbb{N},$$ and $$\mathcal{D}(z_ns(x_n))=\mathcal{D}(s(x_n))\stackrel{t(L^\infty(\Omega))}{\longrightarrow} 0.$$ Consequently, $s(x_n)\stackrel{t(\mathcal{M})}{\longrightarrow} 0$ (see Proposition \ref{plm-spk1}(i)).

Since $E^\bot_\lambda(|x_n|)\leq s(x_n)$ for all $\lambda>0,\ n\in\mathbb{N}$, it follows $E_\lambda^\bot(|x_\alpha|) \stackrel{t(\mathcal{M})}{\longrightarrow} 0$, and therefore $x_n\stackrel{t(\mathcal{M})}{\longrightarrow} 0$ (see  Remark \ref{rem_plm-spk1}).
\end{proof}

If $p_n\in \mathcal{P}_{fin}(\mathcal{M})$ and $p_n\downarrow 0$, then $\mathcal{D}(p_n)\in L^0_+(\Omega,\Sigma,\mu)$ (see (D1)) and $\mathcal{D}(p_n)\downarrow 0$ (see (D2) and D(7)), in particular, $\mathcal{D}(p_n) \stackrel{t(L^\infty(\Omega))}{\longrightarrow} 0$. Hence, Lemma \ref{l4} implies the following
\begin{corollary}
\label{cc1}
If $\{p_n\}_{n=1}^\infty\subset \mathcal{P}_{fin}(\mathcal{M}),\ p_n\downarrow 0$, then $p_n\stackrel{t(\mathcal{M})}{\longrightarrow} 0$.
\end{corollary}

\begin{lemma}
\label{l6}
Let $x\in S(\mathcal{M}),\ p_n,q_n\in\mathcal{P}(\mathcal{M}),\ p_n\uparrow\mathbf{1},\ q_n\uparrow\mathbf{1},\ xp_n,xq_n\in\mathcal{M},\ p_n^\bot,q_n^\bot\in\mathcal{P}_{fin}(\mathcal{M}),\ n\in\mathbb{N}$. If $\delta: \mathcal{M}\rightarrow LS(\mathcal{M})$ is a derivation, then there exists $\widehat{\delta}(x)\in LS(\mathcal{M})$, such that $$t(\mathcal{M})-\lim_{n\rightarrow\infty}\delta(xp_n)=\widehat{\delta}(x)=t(\mathcal{M})-\lim_{n\rightarrow\infty}\delta(xq_n).$$
\end{lemma}
\begin{proof}
For $n<m$ we have $$l(x(p_m-p_n))\sim r(x(p_m-p_n))\leq
p_m-p_n,$$ and therefore, applying Lemma \ref{l3} and properties (D2), (D3),
we obtain
\begin{align*}
\mathcal{D}( & s(\delta(xp_m-xp_n)))=
\mathcal{D}(s(\delta(x(p_m-p_n))))\leq
3\mathcal{D}(s(x(p_m-p_n)))\leq\\ & 3\mathcal{D}(l(x(p_m-p_n))\vee
(p_m-p_n))\leq 6\mathcal{D}(p_m-p_n)\leq 6\mathcal{D}(p_n^\bot).
\end{align*}
Since $\mathcal{D}(p_n^\bot)\in L_+^0(\Omega,\Sigma,\mu)$ (see (D1)) and  $\mathcal{D}(p_n^\bot)\downarrow 0$ (see (D7)) it follows that $\mathcal{D}(p_n^\bot)\stackrel{t(L^\infty(\Omega))}{\longrightarrow} 0$ (see (D7)). Hence, $$\mathcal{D}(s(\delta(xp_m)-\delta(xp_n)))\stackrel{t(L^\infty(\Omega))}{\longrightarrow} 0$$ for  $n,m\rightarrow\infty$. By Lemma \ref{l4}, we have that $(\delta(xp_m)-\delta(xp_n))\stackrel{t(\mathcal{M})}{\longrightarrow} 0$ for $n,m\rightarrow\infty$, i.e. $\{\delta(xp_n)\}_{n=1}^\infty$ is a Cauchy sequence in $(LS(\mathcal{M}),t(\mathcal{M}))$. Consequently, there exists $\widehat{\delta}(x)\in LS(\mathcal{M})$, such that $$t(\mathcal{M})-\lim_{n\rightarrow\infty}\delta(xp_n)=\widehat{\delta}(x).$$

Let us show that $t(\mathcal{M})-\lim_{n\rightarrow\infty}\delta(xq_n)=\widehat{\delta}(x)$.

For each $n\in\mathbb{N}$ we have
\begin{gather*}
\begin{split}
(p_n-q_n) & ((p_n-p_n\wedge q_n)\vee (q_n-p_n\wedge q_n))=\\ &=((p_n-p_n\wedge q_n)-(q_n-p_n\wedge q_n))((p_n-p_n\wedge q_n)\vee (q_n-p_n\wedge q_n))\\ &=(p_n-p_n\wedge q_n)-(q_n-p_n\wedge q_n)=p_n-q_n.
\end{split}
\end{gather*}
Hence, $$r(p_n-q_n)\leq ((p_n-p_n\wedge q_n)\vee
(q_n-p_n\wedge q_n)).$$
Since $$r(x(p_n-q_n))\leq r(p_n-q_n)$$ and $$l(x(p_n-q_n))\sim r(x(p_n-q_n)),$$ it follows
\begin{gather*}
\begin{split}
\mathcal{D}(s(x(p_n-q_n)))&=\mathcal{D}(l(x(p_n-q_n))\vee
r(x(p_n-q_n)))\\ &\stackrel{(D6)}{\leq}
\mathcal{D}(l(x(p_n-q_n)))+\mathcal{D}(r(x(p_n-q_n)))
=2\mathcal{D}(r(x(p_n-q_n)))\cr &\stackrel{(D6)}{\leq}
2\mathcal{D}(p_n-p_n\wedge q_n)+2\mathcal{D}(q_n-p_n\wedge q_n)\leq 4\mathcal{D}(\mathbf{1}-p_n\wedge q_n)\cr
&=4\mathcal{D}(p_n^\bot\vee q_n^\bot)\leq
4(\mathcal{D}(p_n^\bot)+\mathcal{D}(q_n^\bot)).
\end{split}
\end{gather*}

Since (see Lemma \ref{l3}) $$\mathcal{D}(s(\delta(xp_n)-\delta(xq_n)))=\mathcal{D}(s(\delta(x(p_n-q_n))))\leq 3\mathcal{D}(s(x(p_n-q_n))),$$ we have
$$\mathcal{D}(s(\delta(xp_n)-\delta(xq_n)))\leq 12(\mathcal{D}(p_n^\bot)+\mathcal{D}(q_n^\bot))\downarrow 0.$$ By Lemma \ref{l4}, we obtain $$t(\mathcal{M})-\lim_{n\rightarrow\infty}\delta(xq_n)=t(\mathcal{M})-\lim_{n\rightarrow\infty}\delta(xp_n)=\widehat{\delta}(x).$$
\end{proof}

Now, equipped with Lemma \ref{l6}, we may extend any derivation $\delta: \mathcal{M}\rightarrow LS(\mathcal{M})$ up to a derivation $\widehat{\delta}$ from $S(\mathcal{M})$ into $LS(\mathcal{M})$.

For each $x\in S(\mathcal{M})$ there exists a sequence $\{p_n\}\in\mathcal{P}(\mathcal{M})$, such that $p_n\uparrow\mathbf{1},\ p_n^\bot\in\mathcal{P}_{fin}(\mathcal{M}),\ xp_n\in\mathcal{M}$ for all $n\in\mathbb{N}$. By Lemma \ref{l6}, there exists $\widehat{\delta}(x)\in LS(\mathcal{M})$, such that $t(\mathcal{M})-\lim_{n\rightarrow\infty}\delta(xp_n)=\widehat{\delta}(x)$. In addition, the definition of $\widehat{\delta}(x)$ does not depend on a choice of sequence $\{p_n\}_{n\geq 1}$ satisfying the above mentioned property, in particular, $\widehat{\delta}(x)=\delta(x)$ for all $x\in\mathcal{M}$ (in this case, $p_n=\mathbf{1},\ n\in\mathbb{N}$).

\begin{proposition}
\label{p10}
The mapping $\widehat{\delta}$ is a unique derivation from $S(\mathcal{M})$ into $LS(\mathcal{M})$, such that $\widehat{\delta}(x)=\delta(x)$ for all $x\in\mathcal{M}$.
\end{proposition}
\begin{proof}
For $x,y\in S(\mathcal{M})$ select $p_n,q_n\in \mathcal{P}(\mathcal{M}),\ n\in\mathbb{N}$, such that $$p_n\uparrow\mathbf{1},\ q_n\uparrow\mathbf{1},\ p_n^\bot,q_n^\bot\in\mathcal{P}_{fin}(\mathcal{M}),\ xp_n,yq_n\in\mathcal{M}$$ for all $n\in\mathbb{N}$. The sequence of projections $e_n=p_n\wedge q_n$ is increasing, and, in addition, \begin{gather*}
xe_n=xp_ne_n\in\mathcal{M},\ ye_n=yq_ne_n\in\mathcal{M},\\
 e_n^\bot=p_n^\bot\vee q_n^\bot\in\mathcal{P}_{fin}(\mathcal{M}),\mathcal{D}(e_n^\bot)\leq\mathcal{D}(p_n^\bot)+\mathcal{D}(q_n^\bot)\downarrow 0.
\end{gather*}
The last estimate implies the convergence $e_n^\bot\downarrow 0$ (see (D7)), or $e_n\uparrow\mathbf{1}$. By Lemma \ref{l6}, we have
\begin{gather*}
\begin{split}
\widehat{\delta}(x+y)&=t(\mathcal{M})-\lim_{n\rightarrow\infty}\delta((x+y)e_n)=
\\
&=\bigl(t(\mathcal{M})-\lim_{n\rightarrow\infty}\delta(xe_n)\bigl) +\bigl(t(\mathcal{M})-\lim_{n\rightarrow\infty}\delta(ye_n)\bigl)=\widehat{\delta}(x)+\widehat{\delta}(y).
\end{split}
\end{gather*}

Similarly, $\widehat{\delta}(\lambda x)=\lambda\widehat{\delta}(x)$ for all $\lambda\in\mathbb{C}$.

Let us show that $\widehat{\delta}(xy)=\widehat{\delta}(x)y+x\widehat{\delta}(y),x,y\in S(\mathcal{M})$.

Due to polar decomposition $y=u|y|,\ u^*u=r(y)$, we have
$y_n=yE_n(|y|)\in\mathcal{M}$ for all $n\in\mathbb{N}$. Set
$$g_n=\mathbf{1}-r(E_n^\bot(|x|)y_n),\ s_n=g_n\wedge E_n(|y|).$$
Since $$g_n^\bot=r(E_n^\bot(|x|)y_n)\sim
l(E_n^\bot(|x|)y_n)\leq E_n^\bot(|x|),$$ we obtain $$g_n^\bot\preceq
E_n^\bot(|x|).$$ Since $x\in S(\mathcal{M})$, there exists $n_0\in\mathbb{N}$ such that
$E_n^\bot(|x|)\in\mathcal{P}_{fin}(\mathcal{M})$ for all $n\geq
n_0$, and therefore $g_n^\bot\in\mathcal{P}_{fin}(\mathcal{M})$ for
all $n\geq n_0$. The equality
$$y_ng_n=E_n(|x|)y_ng_n+E_n^\bot(|x|)y_ng_n=E_n(|x|)y_ng_n$$
implies that
\begin{gather*}
\begin{split}
E_{n+1}^\bot(|x|)y_{n+1}s_n&=E_{n+1}^\bot(|x|)E_n^\bot(|x|y_{n+1}E_n(|y|))s_n=\\
&=
E_{n+1}^\bot(|x|)(E_n^\bot(|x|)y_nE_n(|y|))s_n=\\&
=E_{n+1}^\bot(|x|)(E_n^\bot(|x|)y_ns_n)=E_{n+1}^\bot(|x|)(E_n^\bot(|x|)y_ng_n)s_n=0,
\end{split}
\end{gather*}
in particular, $$s_n\leq
\mathbf{1}-r(E_{n+1}^\bot(|x|)y_{n+1})=g_{n+1}$$ for all
$n\in\mathbb{N}$. From here and from the inequalities $s_n\leq E_n(|y|)\leq
E_{n+1}(|y|)$ it follows that $s_n\leq s_{n+1}$.

Since $y\in S(\mathcal{M})$, we have $E_n^\bot(|y|)\in \mathcal{P}_{fin}(\mathcal{M})$ for $n\geq n_1$ for some $n_1\geq n_0$. Hence, $$s_n^\bot=g_n^\bot\vee E_n^\bot(|y|)\in\mathcal{P}_{fin}(\mathcal{M})$$ for $n\geq n_1$ and $$\mathcal{D}(s_n^\bot)\leq\mathcal{D}(g_n^\bot)+\mathcal{D}(E_n^\bot(|y|))\leq (\mathcal{D}(E_n^\bot(|x|))+\mathcal{D}(E_n^\bot(|y|)))\downarrow 0,$$ i.e. $s_n^\bot\downarrow 0$ or $s_n\uparrow\mathbf{1}$.

Using Corollary \ref{cc1}, Lemma \ref{l6}, the inclusions $xE_n(|x|)\in\mathcal{M},\ yE_n(|y|)\in\mathcal{M}$ and equalities
\begin{gather*}
\begin{split}
xys_n&=xyE_n(|y_n|)s_n=xy_ns_n=xy_ng_ns_n=
\\
&=xE_n(|x|)y_nq_ns_n=
xE_n(|x|)yE_n(|y|)s_n,
\end{split}
\end{gather*}
we obtain
\begin{gather*}
\begin{split}
\widehat{\delta}(xy)=&t(\mathcal{M})-\lim_{n\rightarrow\infty}\delta(xys_n)= t(\mathcal{M})-\lim_{n\rightarrow\infty}\delta(xE_n(|x|)yE_n(|y|s_n))=
\\&
=t(\mathcal{M})-\lim_{n\rightarrow\infty}\bigl(\delta(xE_n(|x|))ys_n+xE_n(|x|)\delta(ys_n)\bigl)=
\\&
=\bigl(t(\mathcal{M})-\lim_{n\rightarrow\infty}\delta(xE_n(|x|))\bigl)\cdot
\bigl(t(\mathcal{M})-\lim_{n\rightarrow\infty}ys_n\bigl)+\cr &+
\bigl(t(\mathcal{M})-\lim_{n\rightarrow\infty}xE_n(|x|)\bigl)\cdot
\bigl(t(\mathcal{M})-\lim_{n\rightarrow\infty}\delta(ys_n)\bigl)=\widehat{\delta}(x)y+x\widehat{\delta}(y).
\end{split}
\end{gather*}
Consequently, $\widehat{\delta}: S(\mathcal{M})\rightarrow
LS(\mathcal{M})$ is a derivation, such that
$\widehat{\delta}(x)=\delta(x)$ for all $x\in\mathcal{M}$.

Let $\delta_1: S(\mathcal{M})\rightarrow LS(\mathcal{M})$ also be a derivation, for which $\delta_1(x)=\delta(x)$ for all $x\in\mathcal{M}$. If $x\in S(\mathcal{M})$, then $E_n(|x|)\uparrow\mathbf{1},\ xE_n(|x|)\in\mathcal{M},\ n\in\mathbb{N},\ E_n^\bot(|x|)\in\mathcal{P}_{fin}(\mathcal{M})$ for all $n\geq n_3$ for some $n_3\in\mathbb{N}$. Hence, $E_n(|x|)\stackrel{t(\mathcal{M})}{\longrightarrow} \mathbf{1}$ (see Corollary \ref{cc1}).   Since $(LS(\mathcal{M}),t(\mathcal{M}))$ is a topological algebra, it follows that
\begin{gather*}
\begin{split}
\delta_1(x)&=t(\mathcal{M})-\lim_{n\rightarrow\infty}\delta_1(x)E_n(|x|)=
\cr
&=\bigl(t(\mathcal{M})-\lim_{n\rightarrow\infty}\delta_1(xE_n(|x|))\bigl)-
\bigl(t(\mathcal{M})-\lim_{n\rightarrow\infty}x\delta_1(E_n(|x_n|))\bigl)=
\cr
&=\bigl(t(\mathcal{M})-\lim_{n\rightarrow\infty}\delta(xE_n(|x|))\bigl)-\bigl(
t(\mathcal{M})-\lim_{n\rightarrow\infty}x\delta(E_n(|x_n|))\bigl)=\cr &=
\widehat{\delta}(x)-x(t(\mathcal{M})-\lim_{n\rightarrow\infty}\delta(E_n(|x_n|))).
\end{split}
\end{gather*}
Since $\delta(\mathbf{1})=0,\ s(x)=s(-x)$ for $x\in LS(\mathcal{M})$, it follows via Lemma \ref{l3}, that
\begin{gather*}
\begin{split}
\mathcal{D}(s(\delta(E_n(|x|))))&= \mathcal{D}(s(\delta(-E_n(|x|))))=\\
&=\mathcal{D}(s(\delta(\mathbf{1}-E_n(|x|))))\leq 3\mathcal{D}(E_n^\bot(|x|))\downarrow 0.
\end{split}
\end{gather*}  By Lemma \ref{l4}, we obtain $\delta(E_n(|x|))\stackrel{t(\mathcal{M})}{\longrightarrow} 0$, that implies the equality $\delta_1(x)=\widehat{\delta}(x)$.
\end{proof}

Propositions \ref{p9} and \ref{p10} imply the following theorem, which is the main result of this section.

\begin{theorem}
\label{ext}
Let $\mathcal{A}$ be a subalgebra of $LS(\mathcal{M})$, $\mathcal{M}\subset\mathcal{A}$ and let $\delta: \mathcal{A}\rightarrow LS(\mathcal{M})$ be a derivation. Then there exists a unique derivation $\delta_{\mathcal{A}}: LS(\mathcal{M})\rightarrow LS(\mathcal{M})$ such that $\delta_{\mathcal{A}}(x)=\delta(x)$ for all $x\in\mathcal{A}$.
\end{theorem}
\begin{proof}
Since $\mathcal{M}\subset\mathcal{A}$, the restriction $\delta_0$ of the derivation $\delta$ on $\mathcal{M}$ is a well-defined derivation from $\mathcal{M}$ into $LS(\mathcal{M})$. Hence, by Propositions \ref{p9} and \ref{p10}, the mapping $\delta_{\mathcal{A}}=\widetilde{\widehat{\delta}}$ is a unique derivation from $LS(\mathcal{M})$ into $LS(\mathcal{M})$ such that $\delta_{\mathcal{A}}(x)=\delta_0(x)$ for all $x\in\mathcal{M}$. Let us show that $\delta_{\mathcal{A}}(a)=\delta(a)$ for every $a\in\mathcal{A}$. If $a\in\mathcal{A}$, then there exists a sequence $\{z_n\}_{n=1}^\infty\subset\mathcal{P}(\mathcal{Z}(\mathcal{M}))$, such that $z_n\uparrow\mathbf{1}$ and $az_n\in S(\mathcal{M}),\ n\in\mathbb{N}$. Since $z_n\stackrel{t(\mathcal{M})}{\longrightarrow}\mathbf{1}$ (see Proposition \ref{p8}(i)), we have, by Lemma \ref{l1}, $$\delta_{\mathcal{A}}(a)=t(\mathcal{M})-\lim_{n\rightarrow\infty}\delta_{\mathcal{A}}(a)z_n= t(\mathcal{M})-\lim_{n\rightarrow\infty}\delta_{\mathcal{A}}(az_n),$$ and, similarly, $\delta(a)=t(\mathcal{M})-\lim_{n\rightarrow\infty}\delta(az_n)$.

Using the equality $\delta_{\mathcal{A}}(x)=\delta_0(x)=\delta(x)$ for each $x\in\mathcal{M}$, and following the proof of uniqueness of the derivation $\widehat{\delta}$ from Proposition \ref{p10}, we obtain $\delta_{\mathcal{A}}(az_n)=\delta(az_n)$ for all $n\in\mathbb{N}$, that implies the equality $\delta_{\mathcal{A}}(a)=\delta(a)$.
\end{proof}

The following corollary immediately follows from Theorems \ref{main} and \ref{ext}.
\begin{corollary}
\label{c2}
Let $\mathcal{M}$ be a properly infinite von Neumann algebra, $\mathcal{A}$ is a subalgebra in $LS(\mathcal{M})$ and $\mathcal{M}\subset\mathcal{A}$. Then any derivation $\delta: \mathcal{A}\rightarrow LS(\mathcal{M})$ is continuous with respect to the local measure topology $t(\mathcal{M})$.
\end{corollary}

In particular, Corollary \ref{c2} implies that for a properly infinite von Neumann algebra $\mathcal{M}$ any derivation $\delta: S(\mathcal{M})\rightarrow S(\mathcal{M})$ is $t(\mathcal{M})$-continuous.  Note, that in case, when $\mathcal{M}$ is of type $I_\infty$, any derivation of $S(\mathcal{M})$ is inner \cite{AK}, and therefore is automatically continuous with respect to the topology $t(\mathcal{M})$.

\section{Applications to the algebra $S(\mathcal{M},\tau)$ of $\tau$-measurable operators}

Let $\mathcal{M}$ be a semifinite von Neumann algebra acting on Hilbert space $H$, $\tau$ be a faithful normal semifinite trace on $\mathcal{M}$. An operator $x\in S(\mathcal{M})$ with domain $\mathfrak{D}(x)$ is called \emph{$\tau$-measurable} if for any $\varepsilon>0$ there exists a projection $p\in\mathcal{P}(\mathcal{M})$ such that $p(H)\subset \mathfrak{D}(x)$ and $\tau(p^\bot)<\infty$.

The set $S(\mathcal{M},\tau)$ of all $\tau$-measurable operators is a $*$-subalgebra of $S(\mathcal{M})$ such that $\mathcal{M}\subset S(\mathcal{M},\tau)$. If the trace $\tau$ is finite, then $S(\mathcal{M},\tau)=S(\mathcal{M})$. The algebra $S(\mathcal{M},\tau)$ is a noncommutative version of the algebra of all measurable complex functions $f$ defined on $(\Omega,\Sigma,\mu)$, for which $\mu(\{|f|>\lambda\})\rightarrow 0$ for $\lambda\rightarrow\infty$. For each $x\in S(\mathcal{M},\tau)$ it is possible to define the generalized singular value function
$$\mu_t(x)=\inf\{\lambda>0:\tau(E_\lambda^\bot(|x|)<t\}=\inf\{\|x(\mathbf{1}-e)\|_\mathcal{M}:e\in\mathcal{P(M)},\tau(e)<t\},$$
which allows to define and  study a noncommutative version of rearrangement invariant function spaces. For the theory of the latter spaces, we refer to \cite{DDdP},\cite{K-S}.

Let $t_\tau$ be the measure topology \cite{Ne} on $S(\mathcal{M},\tau)$, whose base of neighborhoods of zero is given by $U(\varepsilon,\delta)=\{x\in S(\mathcal{M},\tau):\ \exists p\in \mathcal{P}(\mathcal{M}),\ \tau(p^\bot)\leq\delta,\ xp\in\mathcal{M},\ \|xp\|_{\mathcal{M}}\leq\varepsilon \}$, $\varepsilon>0,\ \delta>0$.

The pair $(S(\mathcal{M},\tau),t_\tau)$ is a complete metrizable topological $*$-algebra. Here, the topology $t_\tau$ majorizes the topology $t(\mathcal{M})$ on $S(\mathcal{M},\tau)$ and, if $\tau$ is a finite trace, the topologies $t_\tau$ and $t(\mathcal{M})$ coincide \cite[\S\S\,3.4,3.5]{MCh}.

However, if $\tau(\mathbf{1})=\infty$, then on $(S(\mathcal{M},\tau),t_\tau)$ topologies $t_\tau$ and $t(\mathcal{M})$ do not coincide in general \cite{CM}. For example, when $\mathcal{M}=L^\infty(\Omega,\Sigma,\mu),\tau(f)=\int_\Omega fd\mu,f\in L^\infty_+(\Omega),$ where $\mu$ is a $\sigma$-finite measure, $\mu(\Omega)=\infty$, the topology $t_\tau$ in $S(L^\infty(\Omega),\tau)$ coincide with the topology of convergence in measure $\mu$, and the topology $t(L^\infty(\Omega))$ is the topology of convergence locally in measure $\mu$ (Section 2), in particular, if $A_n\in\Sigma,\mu(A_n)=\infty,n\in\mathbb{N}$ and $\chi_{A_n}\downarrow 0$, then $\chi_{A_n}\xrightarrow{t(L^\infty(\Omega))} 0$, whereas $\chi_{A_n}\stackrel{t_\tau}{\nrightarrow}0$. See the detailed comparison of topologies $t_\tau$ and $t(\mathcal{M})$ in \cite{CM}.

It is proved in \cite{Ber} that in a properly infinite von Neumann algebra $\mathcal{M}$ any derivation  $\delta: S(\mathcal{M},\tau)\rightarrow S(\mathcal{M},\tau)$ is continuous with respect to the topology $t_\tau$. Corollary \ref{c2} implies that, in this case, derivation $\delta: S(\mathcal{M},\tau)\rightarrow S(\mathcal{M},\tau)$ is continuous with respect to the topology $t(\mathcal{M})$ too.

Now, we give an application of Theorem \ref{ext} to derivations defined on absolutely solid $*$-subalgebras of the algebra $LS(\mathcal{M})$.

Recall \cite{BdPS}, that a $*$-subalgebra $\mathcal{A}$ of $LS(\mathcal{M})$ is called absolutely solid if conditions $x\in LS(\mathcal{M}),\ y\in\mathcal{A},\ |x|\leq |y|$ imply that $x\in\mathcal{A}$.
In \cite[Proposition 5.13]{BdPS} it is proved that, if $\delta$ is a derivation on absolutely solid $*$-subalgebra $\mathcal{A}\supset\mathcal{M}$ and $\delta(x)=[w,x]$ for all $x\in\mathcal{A}$ and some $w\in LS(\mathcal{M})$, then there exists $w_1\in\mathcal{A}$, such that $\delta(x)=[w_1,x]$  for all $x\in\mathcal{A}$, i.e. the derivation $\delta$ is inner on the $*$-subalgebra $\mathcal{A}$. This observation and Theorem \ref{ext} yield our final result

\begin{corollary}
\label{c3}
Suppose that all derivations on the algebra $LS(\mathcal{M})$ are inner and let $\mathcal{A}\supset\mathcal{M}$ be an absolutely solid $*$-subalgebra of $LS(\mathcal{M})$. Then all derivations on $\mathcal{A}$  are inner. In particular, any derivation on the algebras $S(\mathcal{M})$ and $S(\mathcal{M},\tau)$ are inner.
\end{corollary}
The result of Corollary \ref{c3} extends and generalizes \cite[Theorem 6.8]{AK} and \cite[Proposition 5.17]{BdPS}. Note, that the conditions of Corollary \ref{c3} hold, in particular, for properly infinite von Neumann algebras, which do not have direct summand of type $II_\infty$ \cite{AAK},\cite{AK},\cite{BdPS}.

\end{document}